\documentclass[oneside,preprint,12pt,numbers,sort&compress]{elsarticle}
\usepackage{amssymb}
\usepackage{amsthm}
\usepackage{mathrsfs}
\usepackage{amsmath}
\usepackage{lineno}
\usepackage[text={170mm,240mm},centering]{geometry}

\usepackage{xcolor}
\usepackage[colorlinks,allcolors=teal,bookmarksnumbered=true]{hyperref} 
\usepackage[nameinlink,capitalize]{cleveref} 
\crefformat{equation}{#2(#1)#3}
\crefmultiformat{equation}{(#2#1#3)}{ and~(#2#1#3)}{, (#2#1#3)}{ and~(#2#1#3)}

\journal{journal}

\allowdisplaybreaks[4]

\begin{document}
\newtheorem{theorem}{Theorem}[section]
\newtheorem{lemma}[theorem]{Lemma}
\newtheorem{proposition}[theorem]{Proposition}
\newtheorem{remark}[theorem]{Remark}
\newtheorem{definition}[theorem]{Definition}
\newtheorem{Hypothesis}[theorem]{Hypothesis}
\newtheorem{corollary}[theorem]{Corollary}

\numberwithin{equation}{section} 

\begin{frontmatter}
\title{Pullback exponential attractors for 3D non-autonomous Boussinesq equations with fractional Laplacian}


\author[mymainaddress1,mymainaddress2]{Tong Kou}
\ead{kou\_tong2023@163.com}

\author[mymainaddress1,mymainaddress3]{Hui Liu\texorpdfstring{\corref{mycorrespondingauthor}}{}}
\cortext[mycorrespondingauthor]{Corresponding author}
\ead{ss\_liuhui@ujn.edu.cn}

\author[mymainaddress4]{Xiuqing Wang}
\ead{daqingwang@kust.edu.cn}
		
\author[mymainaddress5]{Jie Xin}
\ead{fdxinjie@qfnu.edu.cn}


\address[mymainaddress1]{School of Mathematical Sciences, Qufu Normal University, Qufu 273165, PR China}

\address[mymainaddress2]{MOE-LCSM, School of Mathematics and Statistics, Hunan Normal University, Changsha 410081, PR China}

\address[mymainaddress3]{School of Mathematical Sciences, University of Jinan, Jinan 250022, PR China}

\address[mymainaddress4]{Faculty of Science, Kunming University of Science and Technology, Kunming 650500‌, PR China}

\address[mymainaddress5]{School of Information Engineering, Shandong Youth University of Political Science, Jinan 250103, PR China}


\begin{abstract}
\par Focusing on the pullback asymptotic behavior exhibited by the system, this paper studies the three-dimensional (3D) non-autonomous Boussinesq equations with fractional Laplacian. For $ \alpha > \frac{5}{4} $ and $ \beta > \frac{5}{8} $, by adopting the $ \ell $-trajectory method introduced in \cite{MP2002}, we demonstrate the existence of both a finite-dimensional pullback attractor and a pullback exponential attractor for the process $ \{L(t, \tau)\}_{t \geq \tau} $ that corresponds to \cref{eq0101} in the phase space $ X_{\ell} $. Furthermore, we establish the existence of a pullback exponential attractor for the process $ \{U(t, \tau)\}_{t \geq \tau} $ within the original phase space $ H_\sigma^\alpha(\Omega) \times H^{\beta}(\Omega) $.
\end{abstract}

\end{frontmatter}

\section{Introduction}
The Boussinesq system serves as a classical mathematical model for describing buoyancy-driven fluid motions and has been extensively applied in fields such as atmospheric science and ocean circulation (see, e.g., \cite{Gill1982,Pedlosky1987,CD1999,MB2002}). Coupling the incompressible Navier-Stokes equations with the heat conduction equation, the Boussinesq system accurately describes the interaction between the velocity and temperature fields under gravity, thus providing a critical mathematical framework for investigating complex phenomena such as atmospheric circulation, oceanic convection, and density-driven heavy gas diffusion in industrial processes. 

Over the past several decades, extensive research focused on the Boussinesq system has yielded notable advancements. With the development and refinement of fractional calculus theory, the 3D Boussinesq system incorporating fractional dissipation terms has garnered significant interest within the research domains of fluid mechanics and partial differential equations (PDEs). Notably, while the system might at first seem to be merely a mathematical generalization, it bears concrete relevance to geophysical scenarios. A typical example is the flow regime in the middle atmosphere: as airflows propagate upward, they undergo modifications driven by variations in atmospheric properties—even when the incompressibility approximation and Boussinesq approximation remain valid. Here, atmospheric thinning dampens the effects of kinematic and thermal diffusion, with this anomalous dampening effectively modeled via the spatial fractional Laplacian (see, e.g., \cite{Gill1982}).

In this paper, we consider the following 3D non-autonomous Boussinesq equations with fractional Laplacian:
\begin{equation}\label{eq0101}
\begin{cases}
\partial_{t}u+\nu(-\Delta)^{\alpha}u+(u\cdot\nabla)u+\nabla p = \theta e_3 + f(x,t),\ (x,t)\in \Omega \times \mathbb{R}_{\tau}, \\
\partial_{t}\theta+\kappa(-\Delta)^{\beta}\theta+(u\cdot\nabla)\theta = g(x,t),\ (x,t)\in \Omega \times \mathbb{R}_{\tau}, \\
\nabla\cdot u = 0,\ (x,t)\in \Omega \times \mathbb{R}_{\tau},  \\
u(x,\tau) = u_\tau(x), \quad \theta(x,\tau) = \theta_\tau(x),\ x \in \Omega, 
\end{cases}
\end{equation}
where $\Omega = \mathbb{T}^3$ is defined as \(\mathbb{R}^3/(2\pi\mathbb{Z}^3)\), corresponding to the periodic domain \([-\pi, \pi]^3\); and \(\mathbb{R}_{\tau}\) denotes the interval \([\tau, +\infty)\) with \(\tau \in \mathbb{R}\). $u=(u_1,u_2,u_3)$ denotes the velocity field, while $p$ represents the pressure and $\theta$ stands for the scalar temperature. The constants $\nu>0$ and $\kappa>0$ are the coefficients
of velocity dissipation and thermal diffusion, respectively. Here, \(f = (f_1, f_2, f_3)\) denotes the 3D vector-valued external forcing term for the velocity equation, and \(g\) denotes the scalar external forcing term for the temperature equation, both of which are time-dependent. $e_3=(0,0,1)$. 

Extensive literature has been focused on the investigation of Boussinesq equations in mathematics, with remarkable progress achieved. Specifically, when $\alpha = \beta = 1$, \cref{eq0101} reduces to the classical Boussinesq equations. Global weak solutions, conditional and partial regularity results, and various aspects of its long-time dynamics have been extensively studied. A series of systematic results can be found in \cite{BS2004,CRT2004,WangXM2007,AH2007,GR2020}.

The well-posedness of Boussinesq systems with fractional
dissipation has been studied in different spatial dimensions. For research on the 2D case, please refer to \cite{YX2016,DYZ2020}; for the 3D case, results on well-posedness are available in \cite{BF2017,YJW2018,LY2021,XZ2024}; regarding general $N$-dimensional cases, relevant results are reported in \cite{Yamazaki2015}, in which Yamazaki investigated the $N$-dimensional generalized Boussinesq system in which dissipation and diffusion generalized by fractional Laplacians, demonstrating that under critical dissipation, the solution remains smooth for all times even if the diffusivity is zero. Additionally, the dynamical behaviors of Boussinesq equations with fractional Laplacian have attracted widespread attention. For example, in the 2D subcritical case, Huo and Huang established the global well-posedness of strong solutions and the existence of the global attractor in \cite{HH2016}. Huang, Huo and Jolly established the finite-dimensionality of the global attractor and estimated the number of determining modes in \cite{HHJ2018}. In \cite{AT2023}, Anh and Tinh systematically analyzed the existence, uniqueness, and regularity of global weak solutions for the 3D generalized Boussinesq system with fractional Laplacians, and revealed the asymptotic behavior of weak solutions via attractors based on the Cheskidov-Lu evolutionary system framework. For systems with partial dissipation, we refer to \cite{YJW2018,HS2020,HMS2022}. Related results for Boussinesq systems with damping and stochastic forcing can be found in \cite{LLX2022} and \cite{WH2024,ZCY2024}, respectively. For additional research on the Boussinesq equations, we refer the reader to \cite{XY2013,KC2014,JY2016,BF2017,Yue2023,WLCD2025,LLS2025,LSL2025}.

Global attractors provide a framework for describing the long-time dynamics of dissipative evolution equations (see \cite{Temam1997,CV2002}). However, their defining attraction property does not specify a rate of convergence. Exponential attractors provide compact, positively invariant sets of finite fractal dimension that attract bounded sets at an exponential rate; see \cite{EFNT1994}. Currently, there are three main classical methods for constructing exponential attractors for autonomous dissipative systems: those based on the contraction property of solution differences \cite{EFNT1994}, smoothness \cite{EMZ2000}, and quasi-stability \cite{Chueshov2015}. Furthermore, to ensure that the constructed exponential attractors possess finite fractal dimensions, it is also necessary to verify the Hölder continuity of the semigroup with respect to time—a condition that is often difficult to check when the regularity of solutions is poor.

Similar to the autonomous setting, many scholars have also put forward the notion of pullback exponential attractors. Efendiev, Zelik, and Miranville \cite{EZM2005} utilized the idea of forward attractors to extend the development of exponential attractors from discrete semigroups (see \cite{EMZ2000}) to address non-autonomous problems, and based on the smoothness of evolutionary processes, proposed an explicit construction algorithm for the discrete systems. Furthermore, they also developed exponential attractors for continuous-time processes induced by non-autonomous reaction-diffusion systems. Subsequently, this development was further refined within the pullback framework: \cite{LMR2010,CE2011}, drawing on the existence of fixed bounded pullback absorbing sets, extended the algorithm to continuous-time evolutionary processes, resulting in pullback exponential attractors whose sections are bounded in the past but potentially unbounded in the future. Carvalho and Sonner \cite{CS2013} proved the existence of pullback exponential attractors for asymptotically compact processes under weaker assumptions, which do not require the processes to have strong temporal regularity nor the sections to be uniformly bounded in the past. Moreover, by leveraging the existence of a class of time-dependent absorbing sets, they obtained improved estimates for the fractal dimensions of pullback attractors. You \cite{YouB2021} developed an abstract framework for the construction of pullback exponential attractors based on the $\ell$-trajectory method, and this approach was subsequently applied to the three-dimensional non-autonomous primitive equations in \cite{YouB2023}. Related developments for other classes of non-autonomous dissipative systems can be found in the works of Aouadi \cite{Aouadi2023jmp,Aouadi2024amo,Aouadi2024aa}, where pullback exponential attractors and their robustness or continuity properties were investigated for several plate and elasticity models.

In this paper, we study the pullback dynamics of the 3D non-autonomous Boussinesq system \cref{eq0101} on $\Omega$ for $\alpha>\frac{5}{4}$ and $\beta>\frac{5}{8}$. Under the assumptions on the forcing terms stated below, we establish the existence of pullback attractors and pullback exponential attractors in the phase space $H_\sigma^\alpha(\Omega)\times H^\beta(\Omega)$. The restrictions on $\alpha$ and $\beta$ are used to close the higher-order energy estimates and the estimates for differences of strong solutions required by the $\ell$-trajectory construction. A key difficulty lies in controlling the coupled transport terms in these difference estimates at the regularity levels of the chosen phase space. To this end, we employ the $\ell$-trajectory method of \cite{MP2002} and establish the smoothing property of the trajectory process together with the Lipschitz continuity of the endpoint evaluation map on the relevant trajectory sets. These properties allow us to construct attractors in the trajectory space $X_\ell$ and transfer them to the original phase space.

The remainder of the paper is organized as follows. In \cref{sec02}, we introduce the functional setting and preliminary results and establish global well-posedness of \cref{eq0101} for strong solutions. In \cref{sec03}, we recall the definitions of pullback attractors and pullback exponential attractors and prove the existence of the pullback attractor $\hat{\mathcal{A}}_\ell$ for the trajectory process $\{L(t,\tau)\}_{t\geq\tau}$ in $X_\ell$, following the $\ell$-trajectory approach used in \cite{LY2020,AT2021,YouB2021,YouB2023,AS2024}. In \cref{sec04}, we use the smoothing property of this process to construct a pullback exponential attractor in $X_\ell$. We then establish the existence of pullback attractors and pullback exponential attractors for $\{U(t,\tau)\}_{t\geq\tau}$ in the original phase space $H_\sigma^\alpha(\Omega)\times H^\beta(\Omega)$.

\section{Preliminaries}\label{sec02}
We use $L^r(\Omega)$ to denote the usual Lebesgue spaces. We define 
$$H_{1}=\left\{u\in [L^{2}(\Omega)]^{3}\ \middle|\ \nabla\cdot u = 0, \int_{\Omega}u dx=0\right\}, \quad H_{2}=\left\{\theta\in L^{2}(\Omega)\ \middle|\ \int_{\Omega}\theta dx=0\right\}.$$
Let $\Lambda :=(-\Delta)^{\frac{1}{2}}$ denote the fractional Laplacian. For any $f\in L^{1}(\Omega)$, its Fourier transform is define by
$$ \widehat{f}(k): = \frac{1}{(2\pi)^3} \int_{\Omega} f(x) e^{-ik \cdot x}dx \ \text{, which satisfies} \quad \widehat{\Lambda f}(k) = |k| \widehat{f}(k).$$
In this paper, it is convenient to express $\Lambda^s f$, where $\Lambda^s$ is defined by
\begin{equation*}
\Lambda^s f(x):= \sum_{k \in \mathbb{Z}^3 \setminus \{0\}} |k|^s \widehat{f}(k) \mathrm{e}^{\mathrm{i} k \cdot x} \quad \text{when} \quad f(x)= \sum_{k \in \mathbb{Z}^3 \setminus \{0\}} \widehat{f}(k) \mathrm{e}^{\mathrm{i} k \cdot x}.
\end{equation*}

For $s \in \mathbb{R}$, we define the homogeneous periodic Sobolev space 
\begin{equation*}
H^s(\Omega) = \left\{ f= \sum_{k \in \mathbb{Z}^3 \setminus \{0\}} \widehat{f}(k) \mathrm{e}^{\mathrm{i} k \cdot x}\ \middle|\ \sum_{k \in \mathbb{Z}^3\setminus \{0\}} |k|^{2s} |\widehat{f}(k)|^2 < \infty \right\},
\end{equation*}
where the norm is defined by
\begin{equation*}
\| f \|_{H^s(\Omega)}^2 := (2\pi)^3 \sum_{k \in \mathbb{Z}^3 \setminus \{0\}} |k|^{2s} |\widehat{f}(k)|^2.
\end{equation*}
This norm is consistent with the fractional Laplacian, satisfying the relation
$$
\| f \|_{H^s(\Omega)} = \| \Lambda^s f \|_{L^2(\Omega)}.
$$

For the velocity field, we define the divergence-free Sobolev space
\[
H_\sigma^s(\Omega)
=
\left\{
u\in [H^s(\Omega)]^3
\ \middle|\
\nabla\cdot u=0
\right\}.
\]
Since the zero Fourier mode is excluded in the definition of $H^s(\Omega)$, the zero-mean condition is already incorporated. 

According to spectral theory, the operator $\Lambda$ possesses a sequence of positive eigenvalues \(0 < \lambda_1 \leq \lambda_2 \leq \cdots\) increasing to infinity (i.e., $\lambda_j \to \infty$ as $j \to \infty$) in the space $H_{2}$, together with a corresponding orthonormal sequence of eigenvectors $\{a_j\}_{j=1}^{\infty} \subset D(\Lambda)$ satisfying $\Lambda a_j = \lambda_j a_j$. Similarly, in $H_{1}$, $\Lambda$ possesses a sequence of positive eigenvalues \(0 < \eta_1 \leq \eta_2 \leq \cdots\) increasing to infinity (i.e., $\eta_j \to \infty$ as $j \to \infty$), with a corresponding orthonormal sequence of eigenvectors $\{b_j\}_{j=1}^{\infty} \subset D(\Lambda)$ satisfying $\Lambda b_j = \eta_j b_j$.

Throughout this paper, we assume that
$f\in L^2_{\mathrm{loc}}(\mathbb{R};H_1)$ and
$g\in L^2_{\mathrm{loc}}(\mathbb{R};H_2)$ satisfy
\begin{equation}\label{ass:forcing}
G_f := \sup_{r\in\mathbb{R}}
\int_{r-1}^{r}\|f(s)\|_{L^2(\Omega)}^2\,ds < \infty, \quad
G_g := \sup_{r\in\mathbb{R}}
\int_{r-1}^{r}\|g(s)\|_{L^2(\Omega)}^2\,ds < \infty.
\end{equation}
$C>0$ denotes a generic constant whose value may change from line to line. 

On the periodic domain $\Omega = \mathbb{T}^3$, we apply the Leray projection operator $\mathbb{P}:[L^2(\Omega)]^{3}\to H_{1}$ to the velocity equation of the Boussinesq system. This eliminates the pressure term and yields
\begin{equation}\label{eq0201}
\begin{cases}
\partial_t u + \nu \Lambda^{2\alpha} u + \mathbb{P} \left( (u \cdot \nabla) u \right) = \mathbb{P} (\theta e_3) + \mathbb{P} f(x,t), \\
\partial_t \theta + \kappa \Lambda^{2\beta} \theta + (u \cdot \nabla) \theta = g(x,t), \\
u(x,\tau) = u_\tau(x), \quad \theta(x,\tau) = \theta_\tau(x).
\end{cases}
\end{equation} 

We now recall definitions and results pertaining to pullback attractors and pullback exponential attractors for the process generated by an inﬁnite-dimensional non-autonomous dynamical system, for further details, please refer to \cite{CLR2013,CV2002,Ha1988}.

\par Let $X$ be a Banach space with distance $d_{X}(\cdot,\cdot)$ and let $\mathscr{C}(X)$ denote the set of all continuous mappings from $X$ to $X$.  A process on $X$ is a family $\{ U(t,\tau) : t \geq \tau \} \subset \mathscr{C}(X)$ that satisfies the following properties:  
\begin{enumerate}  
    \item $U(\tau,\tau) = I$ for all $\tau \in \mathbb{R}$;
    \item $U(t,\tau) = U(t,r)U(r,\tau)$ for all $t \geq r \geq \tau$;
    \item The map $(t,\tau,x) \mapsto U(t,\tau)x$ is continuous for all $t \geq \tau$ and $x \in X$.  
\end{enumerate}   

\par We denote by $\text{dist}(A,B)$ the Hausdorff semidistance between $A$ and $B$, defined as
\[
\text{dist}(A,B)=\sup_{a\in A} \inf_{b\in B} d_{X}(a,b),\quad A,B\subset X.
\]

\begin{lemma}[\cite{MN1996,CV2002,MP2002}]\label{ALlem}
Let $ p_1 \in (1, \infty] $ and $ p_2 \in [1, \infty) $ be given exponents. Let $ X $ be a Banach space, and let $ X_0, X_1 $ be separable reflexive Banach spaces satisfying $ X_0 \subset\subset X \subset X_1 $. Then 
$$
Y = \left\{ u \in L^{p_1}(0, \ell; X_0) : u' \in L^{p_2}(0, \ell; X_1) \right\} \subset\subset L^{p_1}(0, \ell; X),
$$
where $ \ell $ is a fixed positive constant.
\end{lemma}

\begin{definition}[\cite{Temam1997,LY2020,AT2021}]
A process $\{U(t,\tau)\}_{t\geq \tau}$ on a Banach space $X$ is called $\tau$-continuous if for all $t \in \mathbb{R}$ and each $u_0 \in X$, the $X$-valued function $\tau \mapsto U(t,\tau)u_0$ (with $\tau \in (-\infty, t]$) is continuous with respect to the norm topology of $X$ and bounded in $X$ on $(-\infty, t]$.
\end{definition}

\begin{definition}[\cite{Temam1997,LY2020,AT2021}]\label{weishudingyi}
Let \(H\) be a separable real Hilbert space. For any nonempty compact subset \(K \subset H\), the fractal dimension of \(K\) is defined as
\[
d_F(K) = \limsup_{\epsilon \to 0^+} \frac{\log(N_\epsilon(K))}{\log\left(\frac{1}{\epsilon}\right)},
\]
where \(N_\epsilon(K)\) denotes the minimum number of open balls in \(H\) of radius \(\epsilon > 0\) needed to cover \(K\).
\end{definition}

\begin{lemma}[\cite{MP2002,LY2020,AT2021}]\label{weishuguji}
Let \(X\) and \(Y\) be metric spaces, and let \(f: X \to Y\) be \(\alpha\)-Hölder continuous on \(A \subset X\); then  
\[d_{F}(f(A)) \leq \frac{1}{\alpha} d_{F}(A).\]  
In particular, fractal dimension does not increase for Lipschitz continuous mappings.
\end{lemma}

\begin{definition}[\cite{EFNT1994,EZM2005,CE2011,CS2013}]\label{lahuizhishudingyi}
Let \(\{U(t,s)\}_{t\geq s}\) be a process on a metric space \(X\). A family \(\mathcal{M} = \{\mathcal{M}(t) : t \in \mathbb{R}\}\) is called a pullback exponential attractor for \(\{U(t,s)\}_{t\geq s}\) if
\begin{enumerate}
\item[(i)] Each \(\mathcal{M}(t) \subset X\) is non-empty and compact (for all \(t \in \mathbb{R}\));
\item[(ii)] \(\mathcal{M}\) is positively semi-invariant, meaning that
\[
U(t,s)\mathcal{M}(s) \subset \mathcal{M}(t) \quad \forall t \geq s;
\]
\item[(iii)] The fractal dimension of \(\mathcal{M}(t)\) is uniformly bounded over all \(t \in \mathbb{R}\);
\item[(iv)] \(\mathcal{M}\) exponentially pullback attracts all bounded subsets of \(X\): specifically, there exists a constant \(\omega > 0\) such that for each \(t \in \mathbb{R}\) and every bounded subset \(B \subset X\),
\[
\lim_{s \to +\infty} e^{\omega s} \operatorname{dist}\left(U(t,t - s)B, \mathcal{M}(t)\right) = 0.\]
\end{enumerate}
\end{definition}

\begin{theorem}
Suppose that $(u_{\tau}, \theta_{\tau})\in H_{1} \times H_{2}$ and $ \left(f, g\right) \in L^2_{loc}(\mathbb{R}; H_{1} \times H_{2})$. For any $\tau\in\mathbb{R}$ and $T\geq \tau$, there exists at least one weak solution $(u,\theta)$ to \cref{eq0101} on $[\tau, T]$ satisfying 
\[
\frac{d}{dt} \int_{\Omega} u \cdot \varphi dx
- \int_{\Omega} u \cdot ( u \cdot \nabla \varphi ) dx
+ \nu \int_{\Omega} ( \Lambda^{\alpha} u ) \cdot ( \Lambda^{\alpha} \varphi ) dx
= \int_{\Omega} \theta e_{3} \cdot \varphi dx
+ \int_{\Omega} f \cdot \varphi dx,
\]
and
\[
\frac{d}{dt} \int_{\Omega} \theta  \psi dx
- \int_{\Omega} \theta ( u \cdot \nabla \psi ) dx
+ \kappa \int_{\Omega} ( \Lambda^{\beta} \theta ) ( \Lambda^{\beta} \psi ) dx
= \int_{\Omega} g \psi dx,
\]
for any $ (\varphi,\psi) \in C_{\sigma}^{\infty} (\Omega) \times C^{\infty}(\Omega) $. Furthermore, for all $\tau\in\mathbb{R}$ and $T\geq\tau$,
\begin{equation*}
u\in L^{\infty}(\tau,T; H_{1})\cap L^{2}(\tau,T; H_\sigma^\alpha(\Omega)),
\quad
\theta\in L^{\infty}(\tau,T; H_{2})\cap L^{2}(\tau,T; H^{\beta}(\Omega)).
\end{equation*}
\end{theorem}

Given that the well-posedness of the Boussinesq equations has been extensively studied in \cite{Ye2015,JY2016}, drawing on \cite{Temam1997}, we derive the following well-posedness result.

\begin{theorem}\label{thm0203}
Suppose that $ \alpha > \frac{5}{4} $, $ \beta > \frac{5}{8} $, $(u_{\tau}, \theta_{\tau})\in H_\sigma^\alpha(\Omega) \times H^{\beta}(\Omega)$, and $ \left(f, g\right) \in L^2_{loc}(\mathbb{R}; H_{1} \times H_{2})$. For all $\tau\in\mathbb{R}$ and $T\geq\tau$, there exists a unique strong solution $(u,\theta)$ to \cref{eq0101} on $[\tau, T]$ satisfying
\[
(u,\theta)\in
C([\tau,T];H_\sigma^\alpha(\Omega)\times H^\beta(\Omega))
\cap
L^2(\tau,T;H_\sigma^{2\alpha}(\Omega)\times H^{2\beta}(\Omega)),
\]
and
\[
(u_t,\theta_t)\in L^2(\tau,T;H_1\times H_2).
\]
Moreover, this strong solution depends continuously on the initial data $(u_{\tau}, \theta_{\tau})$ with respect to the topology of $H_\sigma^\alpha(\Omega) \times H^{\beta}(\Omega)$.
\end{theorem}


\begin{corollary}
Suppose that \(\alpha > \frac{5}{4}\), \(\beta > \frac{5}{8}\), and \((f, g) \in L^2_{\text{loc}}(\mathbb{R}; H_1 \times H_2)\), and let \((u_{\tau, m}, \theta_{\tau, m}) \to (u_\tau, \theta_\tau)\) in $H_\sigma^\alpha(\Omega) \times H^\beta(\Omega)$ as $m \to \infty$ (where \((u_{\tau, m}, \theta_{\tau, m}) \in H_\sigma^\alpha(\Omega) \times H^\beta(\Omega)\) is a sequence of initial data). For all $\tau \in \mathbb{R}$ and $T \geq \tau$, let \((u_m(t), \theta_m(t))\) be the unique strong solution to \cref{eq0101} on $[\tau, T]$ corresponding to \((u_{\tau, m}, \theta_{\tau, m})\). Suppose there exists a subsequence $\{(u_{m_k}(t), \theta_{m_k}(t))\}_{k=1}^\infty$ that converges $(\ast)$-weakly to some function \((u(t), \theta(t))\) in space $\{(u(t), \theta(t)) \in L^{\infty}(\tau,T; H_\sigma^{\alpha}(\Omega)\times H^{\beta}(\Omega))
\cap 
L^{2}(\tau,T; H_\sigma^{2\alpha}(\Omega)\times H^{2\beta}(\Omega)); (u_t, \theta_t) \in L^{2}(\tau, T; L^{2}(\Omega)\times L^{2}(\Omega))\}$. Then \((u(t), \theta(t))\) is the unique strong solution to \cref{eq0101} on $[\tau, T]$ with $(u(\tau), \theta(\tau)) = (u_\tau, \theta_\tau)$.
\end{corollary}

\section{Existence of solutions and priori estimates}\label{sec03}
In this section, we employ the $\ell$-trajectory method from \cite{MN1996} to establish the existence of the pullback attractor for the considered system. Using the well-posedness result for strong solutions in \cref{thm0203}, we define a family of continuous processes $\{U(t, \tau)\}_{t\geq\tau}$ by
\[
U(t, \tau)(u_\tau, \theta_\tau) 
= (u(t), \theta(t)) 
:= \bigl(u(t, \tau; (u_\tau, \theta_\tau)), \theta(t, \tau; (u_\tau, \theta_\tau))\bigr)
\]
generated by \cref{eq0101} in $H_\sigma^\alpha(\Omega) \times H^{\beta}(\Omega)$ for all $t \geq \tau$. This process is $(H_\sigma^\alpha(\Omega)\times H^{\beta}(\Omega),H_\sigma^\alpha(\Omega)\times H^{\beta}(\Omega))$-continuous, where $(u(t), \theta(t))$ denotes the strong solution to \cref{eq0101} with $(u_\tau, \theta_\tau)\in H_\sigma^\alpha(\Omega)\times H^{\beta}(\Omega)$. Using \cref{thm0203} together with $(f, g) \in L^2_{\text{loc}}(\mathbb{R}; H_{1} \times H_{2})$, we conclude that the process $\{U(t, \tau)\}_{t\geq\tau}$ is $\tau$-continuous.

\cref{thm0203} implies that the solution to \eqref{eq0101} with initial data in \(H_\sigma^{\alpha}(\Omega) \times H^{\beta}(\Omega)\) is unique. Thus, for any \(\tau \in \mathbb{R}\) and initial data \((u_{\tau}, \theta_{\tau}) \in H_\sigma^{\alpha}(\Omega) \times H^{\beta}(\Omega)\), there exists at most one strong solution on \([\tau, \tau+\ell]\) with the given initial data. We denote by \(C([\tau, \tau+\ell]; H_{1} \times H_{2})\) the space of strongly continuous functions from \([\tau, \tau+\ell]\) to the Banach space \(H_{1} \times H_{2}\). 

We next introduce the $\ell$-trajectory space $X_\ell$, defined as the set of all solution trajectories $\chi(s,\tau; (u_\tau, \theta_\tau)) = (u, \theta)\big|_{[\tau, \tau + \ell]} $, equipped with the topology induced by $ L^2(\tau, \tau + \ell; H_\sigma^{\alpha}(\Omega) \times H^{\beta}(\Omega))$.

The mapping $e_{t}:X_{\ell}\to H_\sigma^{\alpha}(\Omega) \times H^{\beta}(\Omega)$ is defined by
\[
e_{t}(\chi(s,\tau;(u_{\tau},\theta_{\tau})))
=\chi(\tau+t\ell,\tau;(u_{\tau},\theta_{\tau}))
\]
for any $\chi(s,\tau;(u_{\tau},\theta_{\tau}))\in X_{\ell}$ and $t \in [0,1]$.

The operators $L(t, \tau): X_{\ell} \to X_{\ell}$ are defined by
\begin{align*}
  L(t, \tau)\chi(s, \tau; (u_{\tau}, \theta_{\tau})) 
   & = \bigl(u, \theta\bigr)\bigl(t + s - \tau, \tau; (u_{\tau}, \theta_{\tau})\bigr) \\
   & = U\bigl(t + s - \tau, \tau + \ell \bigr)e_1\bigl(\chi(s, \tau; (u_{\tau}, \theta_{\tau}))\bigr) \\
   & = \chi\bigl(t + s - \tau, \tau; (u_{\tau}, \theta_{\tau})\bigr),
     \quad s \in [\tau, \tau + \ell]
\end{align*}
for any $\chi(s,\tau;(u_{\tau},\theta_{\tau}))\in X_{\ell}$, where $(u(t, \tau; (u_\tau, \theta_\tau)), \theta(t, \tau; (u_\tau, \theta_\tau)))$ is the unique solution of \cref{eq0101} with initial data $(u_{\tau}, \theta_{\tau})$. Hence, it follows that the family $\{L(t,\tau)\}_{t\geq \tau}$ forms a process on $X_{\ell}$.

In this paper, denote by \( \mathcal{D}_{\ell} \) the collection of all nonempty bounded subsets of \( X_{\ell} \), and by \( \mathcal{D} \) the collection of all nonempty bounded subsets of $H_\sigma^{\alpha}(\Omega) \times H^{\beta}(\Omega)$.

\subsection{Uniform estimates and pullback absorbing sets}\label{subsec0301}

We now derive prior estimates for solutions to \cref{eq0101}.
\begin{lemma}\label{lem0301}
Assume that \cref{ass:forcing} holds. For any $\tau \in \mathbb{R}$, there exists a positive constant $\rho_1$ such that for every bounded subset $B_{\ell} \in \mathcal{D}_{\ell}$, there exists a time $\tau_1 =\tau_1(B_{\ell}) \geq 0 $, for all strong solutions to \cref{eq0101} with short trajectory $\chi(s,\tau; (u_\tau, \theta_\tau)) \in B_{\ell}$, we have 
\[
\| u(t) \|_{H_{1}}^{2}
+ \| \theta(t) \|_{H_{2}}^{2}
+ \int_{0}^{\ell} \left(
\| \Lambda^\alpha u(t+\zeta)\|_{L^2(\Omega)}^2 + \| \Lambda^\beta \theta(t+\zeta)\|_{L^2(\Omega)}^2 \right) d\zeta
\leq \rho_1 
\]
for all $t-\tau \geq \tau_1$.
\end{lemma}
\begin{proof}
Multiplying $\eqref{eq0101}_{2}$ by $\theta$ and applying Young's inequality, we get that
\begin{align*}
  \frac{1}{2}\frac{d}{dt}\|\theta(t)\|_{L^2(\Omega)}^2 + \kappa\| \Lambda^\beta \theta(t)\|_{L^2(\Omega)}^2
   & \leq \frac{\kappa\lambda_1^{2\beta}}{2} \|\theta(t)\|_{L^2(\Omega)}^2 + \frac{1}{2\kappa\lambda_1^{2\beta}} \|g(t)\|_{L^2(\Omega)}^2 \\
   & \leq \frac{\kappa}{2} \| \Lambda^\beta \theta(t)\|_{L^2(\Omega)}^2 + \frac{1}{2\kappa\lambda_1^{2\beta}} \|g(t)\|_{L^2(\Omega)}^2,
\end{align*}
since the Poincar\'{e} inequality
$$\| \Lambda^\beta \theta(t)\|_{L^2(\Omega)}^2 \geq \lambda_1^{2\beta}\|\theta(t)\|_{L^2(\Omega)}^2. $$
So we have
\begin{align}\label{eq0301}
  \frac{d}{dt}\|\theta(t)\|_{L^2(\Omega)}^2 + \frac{\kappa\lambda_1^{2\beta}}{2} \|\theta(t)\|_{L^2(\Omega)}^2 + \frac{\kappa}{2}\| \Lambda^\beta \theta(t)\|_{L^2(\Omega)}^2 
   & \leq \frac{1}{\kappa\lambda_1^{2\beta}} \|g(t)\|_{L^2(\Omega)}^2,
\end{align}
Gronwall's inequality subsequently yields that
\begin{align}\label{eq0302} 
  \|\theta(t)\|_{L^2(\Omega)}^2
   & \leq e^{\frac{\kappa\lambda_{1}^{2\beta}}{2}(\tau+\zeta-t)} \|\theta(\tau+\zeta)\|_{L^2(\Omega)}^2 + \frac{1}{\kappa\lambda_{1}^{2\beta}} \int_{\tau+\zeta}^{t} \|g(s)\|_{L^2(\Omega)}^2 e^{\frac{\kappa\lambda_{1}^{2\beta}}{2}(s-t)} ds.
\end{align}
Integrating both sides of \cref{eq0302} with respect to $\zeta$ between $0$ and $\ell$ gives 
\begin{equation}\label{eq0303}
\begin{aligned}
   \|\theta(t)\|_{L^2(\Omega)}^2
   \leq & \frac{1}{\ell} \int_{0}^{\ell} e^{\frac{\kappa\lambda_{1}^{2\beta}}{2}(\tau+\zeta-t)} \|\theta(\tau+\zeta)\|_{L^2(\Omega)}^2 d\zeta \\
   & + \frac{1}{\kappa\lambda_{1}^{2\beta}\ell} \int_{0}^{\ell} \int_{\tau+\zeta}^{t} \|g(s)\|_{L^2(\Omega)}^2 e^{\frac{\kappa\lambda_{1}^{2\beta}}{2}(s-t)} ds d\zeta \\
   \leq & \frac{1}{\ell} e^{\frac{\kappa\lambda_{1}^{2\beta}}{2}(\tau+\ell-t)}  \int_{0}^{\ell} \|\theta(\tau+\zeta)\|_{L^2(\Omega)}^2 d\zeta + \frac{1}{\kappa\lambda_{1}^{2\beta}} \left(1+\frac{2}{\kappa\lambda_{1}^{2\beta}}\right)G_g.
\end{aligned}
\end{equation}
Note that
\begin{align*}
   \int_{\tau+\zeta}^{t} \|g(s)\|_{L^2(\Omega)}^2 e^{\frac{\kappa\lambda_{1}^{2\beta}}{2}(s-t)} ds
   \leq & e^{-\frac{\kappa\lambda_{1}^{2\beta}}{2}t} \sum_{n=0}^{+\infty} \left(\int_{t-(n+1)}^{t-n} \|g(s)\|_{L^2(\Omega)}^2 e^{\frac{\kappa\lambda_{1}^{2\beta}}{2}s} ds\right) \\
   \leq & e^{-\frac{\kappa\lambda_{1}^{2\beta}}{2}t} \sum_{n=0}^{+\infty} e^{\frac{\kappa\lambda_{1}^{2\beta}}{2}(t-n)} \left(\int_{t-(n+1)}^{t-n} \|g(s)\|_{L^2(\Omega)}^2 ds\right) \\
   \leq & \left(1+\frac{2}{\kappa\lambda_{1}^{2\beta}}\right)G_g,
\end{align*}
where
$$
G_g:=\sup_{r \in \mathbb{R}}\int_{r-1}^{r}\|g(s)\|_{L^2(\Omega)}^2 ds< +\infty.
$$
For \cref{eq0303}, there exists a constant $\rho_{1}^{(1)} > 0$ such that for any $B_{\ell} \in \mathcal{D}_{\ell}$, there exists $\tau_{1}^{(1)} = \tau_{1}^{(1)}(B_{\ell}) \geq 0$ with the property that 
\begin{equation}\label{eq0317}
    \|\theta(t)\|_{L^2(\Omega)}^2
   \leq 
   \rho_{1}^{(1)},
\end{equation}
\[   
   \rho_{1}^{(1)}
   :=  
   \frac{2}{\kappa\lambda_{1}^{2\beta}} \left(1+\frac{2}{\kappa\lambda_{1}^{2\beta}}\right)G_g
\]
for all $t \in \mathbb{R}$ with $t - \tau \geq \tau_{1}^{(1)}$. Additionally, integrating \cref{eq0301} both sides between $t$ and $t+\ell$ gives
\begin{equation}\label{eq0307}
\begin{aligned}
    & \frac{\kappa\lambda_1^{2\beta}}{2} \int_{0}^{\ell} \|\theta(t+\zeta)\|_{L^2(\Omega)}^2 d\zeta 
    + 
    \frac{\kappa}{2} \int_{0}^{\ell} \| \Lambda^\beta \theta(t+\zeta)\|_{L^2(\Omega)}^2 d\zeta \\
    \leq & \frac{1}{\kappa\lambda_1^{2\beta}} \int_{t}^{t+\ell} \|g(\zeta)\|_{L^2(\Omega)}^2 d\zeta + \|\theta(t)\|_{L^2(\Omega)}^2 \\
    \leq & \frac{1}{\kappa\lambda_1^{2\beta}} \left(1+[\ell]\right)G_g
    + \frac{1}{\ell} e^{\frac{\kappa\lambda_{1}^{2\beta}}{2}(\tau+\ell-t)} \int_{0}^{\ell} \|\theta(\tau+\zeta)\|_{L^2(\Omega)}^2 d\zeta 
    + \frac{1}{\kappa\lambda_{1}^{2\beta}} \left(1+\frac{2}{\kappa\lambda_{1}^{2\beta}}\right)G_g \\
    \leq & \frac{1}{\ell} e^{\frac{\kappa\lambda_{1}^{2\beta}}{2}(\tau+\ell-t)} \int_{0}^{\ell} \|\theta(\tau+\zeta)\|_{L^2(\Omega)}^2 d\zeta 
    + \frac{1}{\kappa\lambda_{1}^{2\beta}} \left(2+\frac{2}{\kappa\lambda_{1}^{2\beta}}+[\ell]\right)G_g,
\end{aligned}
\end{equation}
where $[\ell]$ denotes the floor function of $\ell$(i.e., the greatest integer less than or equal to $\ell$).
With respect to \cref{eq0307}, there exists a constant $\rho_{1}^{(2)} > 0$ such that for any $B_{\ell} \in \mathcal{D}_{\ell}$, there exists $\tau_{1}^{(2)} = \tau_{1}^{(2)}(B_{\ell}) \geq 0$ ensuring that 
\begin{equation}\label{eq0316}
    \int_{0}^{\ell} \| \Lambda^\beta \theta(t+\zeta)\|_{L^2(\Omega)}^2 d\zeta
    \leq 
    \rho_{1}^{(2)},
\end{equation}
\[
   \rho_{1}^{(2)}
   :=
   \frac{4}{\kappa\lambda_{1}^{2\beta}\min\{\kappa\lambda_1^{2\beta},\kappa\}}
   \left(2+\frac{2}{\kappa\lambda_{1}^{2\beta}}+[\ell]\right)G_g
\]
holds for all $t \in \mathbb{R}$ whenever $t - \tau \geq \tau_{1}^{(2)}$.  

Taking the inner product of $\eqref{eq0101}_{1}$ with $u$ and the Poincar\'{e} inequality
$$\| \Lambda^\alpha u(t)\|_{L^2(\Omega)}^2 \geq \eta_1^{2\alpha}\|u(t)\|_{L^2(\Omega)}^2, $$
we obtain
\begin{align*}
   \frac{1}{2}\frac{d}{dt}\|u(t)\|_{L^2(\Omega)}^2 + \nu\|\Lambda^\alpha u(t)\|_{L^2(\Omega)}^2
   & \leq \frac{\nu\eta_1^{2\alpha}}{2} \|u(t)\|_{L^2(\Omega)}^2 + \frac{1}{\nu\eta_1^{2\alpha}} \|\theta(t)\|_{L^2(\Omega)}^2 + \frac{1}{\nu\eta_1^{2\alpha}} \|f(t)\|_{L^2(\Omega)}^2 \\
   & \leq \frac{\nu}{2} \|\Lambda^\alpha u(t)\|_{L^2(\Omega)}^2 + \frac{1}{\nu\eta_1^{2\alpha}} \left( \|\theta(t)\|_{L^2(\Omega)}^2 +    \|f(t)\|_{L^2(\Omega)}^2 \right).
\end{align*}
Therefore, we get
\begin{align}\label{eq0304}
   \frac{d}{dt}\|u(t)\|_{L^2(\Omega)}^2 + \frac{\nu\eta_1^{2\alpha}}{2} \|u(t)\|_{L^2(\Omega)}^2 + \frac{\nu}{2} \|\Lambda^\alpha u(t)\|_{L^2(\Omega)}^2
   & \leq \frac{2}{\nu\eta_1^{2\alpha}} \left( \|\theta(t)\|_{L^2(\Omega)}^2 +    \|f(t)\|_{L^2(\Omega)}^2 \right).
\end{align}
Using Gronwall's inequality and \cref{eq0303}, we obtain
\begin{align*} 
   \|u(t)\|_{L^2(\Omega)}^2
   \leq 
   & e^{\frac{\nu\eta_1^{2\alpha}}{2}(\tau+\zeta-t)} \|u(\tau+\zeta)\|_{L^2(\Omega)}^2 
   + \frac{2}{\nu\eta_1^{2\alpha}} \int_{\tau+\zeta}^{t} \|\theta(s)\|_{L^2(\Omega)}^2 e^{\frac{\nu\eta_1^{2\alpha}}{2}(s-t)} ds \\
   &+ \frac{2}{\nu\eta_1^{2\alpha}} \int_{\tau+\zeta}^{t}  \|f(s)\|_{L^2(\Omega)}^2 e^{\frac{\nu\eta_1^{2\alpha}}{2}(s-t)} ds \\
   \leq 
   & e^{\frac{\nu\eta_1^{2\alpha}}{2}(\tau+\zeta-t)} \|u(\tau+\zeta)\|_{L^2(\Omega)}^2 
   + \frac{2}{\nu\eta_1^{2\alpha}} \int_{\tau+\zeta}^{t}  \|f(s)\|_{L^2(\Omega)}^2 e^{\frac{\nu\eta_1^{2\alpha}}{2}(s-t)} ds \\
   &
   + \frac{2}{\nu\eta_1^{2\alpha}} \|\theta(\tau+\zeta)\|_{L^2(\Omega)}^2  \int_{\tau+\zeta}^{t} e^{\frac{\kappa\lambda_{1}^{2\beta}}{2}(\tau+\zeta-s)} e^{\frac{\nu\eta_1^{2\alpha}}{2}(s-t)} ds \\
   &
   + \frac{2}{\nu\eta_1^{2\alpha}\kappa\lambda_{1}^{2\beta}} \int_{\tau+\zeta}^{t} \left( \int_{\tau+\zeta}^{s} \|g(\sigma)\|_{L^2(\Omega)}^2 e^{\frac{\kappa\lambda_{1}^{2\beta}}{2}(\sigma-s)} d\sigma \right) e^{\frac{\nu\eta_1^{2\alpha}}{2}(s-t)} ds.
\end{align*}
Note that
\begin{align*}
   \int_{\tau+\zeta}^{t}  \|f(s)\|_{L^2(\Omega)}^2 e^{\frac{\nu\eta_1^{2\alpha}}{2}(s-t)} ds
   \leq \left(1+\frac{2}{\nu\eta_1^{2\alpha}}\right)G_f,
\end{align*}
where
$$G_f:=\sup_{r \in \mathbb{R}}\int_{r-1}^{r}\|f(s)\|_{L^2(\Omega)}^2 ds< +\infty.$$

\noindent{\bfseries Case 1}: $\nu\eta_1^{2\alpha} \neq \kappa\lambda_1^{2\beta}$. In addition, we get
\[
\int_{\tau+\zeta}^{t} e^{\frac{\kappa\lambda_{1}^{2\beta}}{2}(\tau+\zeta-s)} e^{\frac{\nu\eta_1^{2\alpha}}{2}(s-t)} ds = \frac{2}{\nu\eta_1^{2\alpha} - \kappa\lambda_1^{2\beta}} \left( e^{\frac{\kappa\lambda_1^{2\beta}(\tau + \zeta - t)}{2}} - e^{\frac{\nu\eta_1^{2\alpha}(\tau + \zeta - t)}{2}} \right).
\]
Hence, when  \(\nu\eta_1^{2\alpha} \neq \kappa\lambda_1^{2\beta}\) is satisfied, we get
\begin{equation}\label{eq0305}
\begin{aligned} 
   \|u(t)\|_{L^2(\Omega)}^2
   \leq 
   & e^{\frac{\nu\eta_1^{2\alpha}}{2}(\tau+\zeta-t)} \|u(\tau+\zeta)\|_{L^2(\Omega)}^2 
    \\
   &
   + \frac{4}{\nu\eta_1^{2\alpha}(\nu\eta_1^{2\alpha} - \kappa\lambda_1^{2\beta})} \|\theta(\tau+\zeta)\|_{L^2(\Omega)}^2 
   \left( e^{\frac{\kappa\lambda_1^{2\beta}(\tau + \zeta - t)}{2}} - e^{\frac{\nu\eta_1^{2\alpha}(\tau + \zeta - t)}{2}} \right) \\
   &
   + \frac{4}{\nu^{2}\eta_1^{4\alpha}\kappa\lambda_{1}^{2\beta}}
   \left(1+\frac{2}{\kappa\lambda_{1}^{2\beta}}\right)
   G_g
   + \frac{2}{\nu\eta_1^{2\alpha}} \left(1+\frac{2}{\nu\eta_1^{2\alpha}}\right)G_f;
\end{aligned}
\end{equation}
Integrating both sides above with respect to $\zeta$ from $0$ to $\ell$, we can derive that
\begin{equation}\label{eq0308}
\begin{aligned} 
   \|u(t)\|_{L^2(\Omega)}^2
   \leq 
   & \frac{1}{\ell} e^{\frac{\nu\eta_1^{2\alpha}}{2}(\tau+\ell-t)}  \int_{0}^{\ell} \|u(\tau+\zeta)\|_{L^2(\Omega)}^2 d\zeta
    \\
   &
   + \frac{4}{\nu\eta_1^{2\alpha}|\nu\eta_1^{2\alpha} - \kappa\lambda_1^{2\beta}|}  
   \frac{1}{\ell}
   e^{\frac{\min\{\nu\eta_1^{2\alpha},\kappa\lambda_1^{2\beta}\}}{2}(\tau+\ell-t)}
   \int_{0}^{\ell} 
   \|\theta(\tau+\zeta)\|_{L^2(\Omega)}^2  d\zeta \\
   &
   +\frac{4}{\nu^{2}\eta_1^{4\alpha}\kappa\lambda_{1}^{2\beta}}
   \left(1+\frac{2}{\kappa\lambda_{1}^{2\beta}}\right)
   G_g
   + \frac{2}{\nu\eta_1^{2\alpha}} \left(1+\frac{2}{\nu\eta_1^{2\alpha}}\right)G_f.
\end{aligned}
\end{equation}
For \cref{eq0308}, there exists a constant $\rho_{1}^{(3)} > 0$ such that for any $B_{\ell} \in \mathcal{D}_{\ell}$, there is a $\tau_{1}^{(3)} = \tau_{1}^{(3)}(B_{\ell}) \geq 0$ where  
\begin{equation}\label{eq0315}
    \|u(t)\|_{L^2(\Omega)}^2 
    \leq 
    \rho_{1}^{(3)},
\end{equation}
\[
    \rho_{1}^{(3)} 
    := 
    \frac{12}{\nu^{2}\eta_1^{4\alpha}\kappa\lambda_{1}^{2\beta}}
    \left(1+\frac{2}{\kappa\lambda_{1}^{2\beta}}\right)
    G_g
    + 
    \frac{6}{\nu\eta_1^{2\alpha}} \left(1+\frac{2}{\nu\eta_1^{2\alpha}}\right)G_f
\]  
for all $t \in \mathbb{R}$ satisfying $t - \tau \geq \tau_{1}^{(3)}$. By integrating both sides of $\eqref{eq0304}$ from $t$ to $t+\ell$, we have
\begin{equation}\label{eq0309}
\begin{aligned}
  &\frac{\nu\eta_1^{2\alpha}}{2} \int_{0}^{\ell} \|u(t+\zeta)\|_{L^2(\Omega)}^2 d\zeta + \frac{\nu}{2} \int_{0}^{\ell} \| \Lambda^\alpha u(t+\zeta)\|_{L^2(\Omega)}^2 d\zeta \\
  \leq 
  &
  \frac{2}{\nu\eta_1^{2\alpha}} \int_{t}^{t+\ell} \|\theta(\zeta)\|_{L^2(\Omega)}^2 d\zeta 
  +\frac{2}{\nu\eta_1^{2\alpha}} \int_{t}^{t+\ell} \|f(\zeta)\|_{L^2(\Omega)}^2 d\zeta 
  + \|u(t)\|_{L^2(\Omega)}^2 \\
  \leq 
  & 
  \frac{4}{\nu\eta_1^{2\alpha}\kappa\lambda_1^{2\beta}} \frac{1}{\ell} e^{\frac{\kappa\lambda_{1}^{2\beta}}{2}(\tau+\ell-t)} \int_{0}^{\ell} \|\theta(\tau+\zeta)\|_{L^2(\Omega)}^2 d\zeta
  + \frac{1}{\ell} e^{\frac{\nu\eta_1^{2\alpha}}{2}(\tau+\ell-t)}  \int_{0}^{\ell} \|u(\tau+\zeta)\|_{L^2(\Omega)}^2 d\zeta \\
  &
   + \frac{4}{\nu\eta_1^{2\alpha}|\nu\eta_1^{2\alpha} - \kappa\lambda_1^{2\beta}|}  
   \frac{1}{\ell}
   e^{\frac{\min\{\nu\eta_1^{2\alpha},\kappa\lambda_1^{2\beta}\}}{2}(\tau+\ell-t)}
   \int_{0}^{\ell} 
   \|\theta(\tau+\zeta)\|_{L^2(\Omega)}^2  d\zeta \\
  &
  + \frac{2}{\nu\eta_1^{2\alpha}\kappa\lambda_{1}^{2\beta}} 
  \left(1+\frac{2}{\kappa\lambda_{1}^{2\beta}}\right)
  \left(\ell+\frac{2}{\nu\eta_1^{2\alpha}}\right)
  G_g
  + \frac{2}{\nu\eta_1^{2\alpha}} \left(2+[\ell]+\frac{2}{\nu\eta_1^{2\alpha}}\right)G_f.
\end{aligned}
\end{equation}
For \cref{eq0309}, there exists a constant $\rho_{1}^{(4)} > 0$ such that for any $B_{\ell} \in \mathcal{D}_{\ell}$, there exists $\tau_{1}^{(4)} = \tau_{1}^{(4)}(B_{\ell}) \geq 0$ where   
\begin{equation}\label{eq0314}
    \int_{0}^{\ell} \| \Lambda^\alpha u(t+\zeta)\|_{L^2(\Omega)}^2 d\zeta
    \leq 
    \rho_{1}^{(4)},
\end{equation} 
with
\[
    \rho_{1}^{(4)} 
    := 
    \frac{16}{\nu\eta_1^{2\alpha}\min\{\nu\eta_1^{2\alpha},\nu\}}
    \left(
    \frac{1}{\kappa\lambda_{1}^{2\beta}} 
    \left(1+\frac{2}{\kappa\lambda_{1}^{2\beta}}\right)
    \left(\ell+\frac{2}{\nu\eta_1^{2\alpha}}\right)
    G_g
    + 
    \left(2+[\ell]+\frac{2}{\nu\eta_1^{2\alpha}}\right)G_f\right)
\]
for all $t \in \mathbb{R}$ satisfying $t - \tau \geq \tau_{1}^{(4)}$.

\noindent{\bfseries Case 2}: $\nu\eta_1^{2\alpha} = \kappa\lambda_1^{2\beta}$. By a direct calculation, we obtain
\[
\int_{\tau+\zeta}^{t} e^{\frac{\kappa\lambda_{1}^{2\beta}}{2}(\tau+\zeta-s)} e^{\frac{\nu\eta_1^{2\alpha}}{2}(s-t)} ds = (t - \tau - \zeta)e^{\frac{\nu\eta_1^{2\alpha}}{2}(\tau+\zeta - t)}.
\]
When \(\nu\eta_1^{2\alpha} = \kappa\lambda_1^{2\beta}\) holds, the result follows as 
\begin{equation}\label{eq0306}
\begin{aligned} 
   \|u(t)\|_{L^2(\Omega)}^2
   \leq 
   & e^{\frac{\nu\eta_1^{2\alpha}}{2}(\tau+\zeta-t)} \|u(\tau+\zeta)\|_{L^2(\Omega)}^2 
   \\
   &
   + \frac{2}{\nu\eta_1^{2\alpha}} \|\theta(\tau+\zeta)\|_{L^2(\Omega)}^2  (t - \tau - \zeta)e^{\frac{\nu\eta_1^{2\alpha}}{2}(\tau+\zeta - t)} \\
   &
   + \frac{4}{\nu^{3}\eta_1^{6\alpha}} 
   \left(1+\frac{2}{\nu\eta_1^{2\alpha}}\right)G_g
   + \frac{2}{\nu\eta_1^{2\alpha}} \left(1+\frac{2}{\nu\eta_1^{2\alpha}}\right)G_f.
\end{aligned}
\end{equation}
Integrating both sides above with respect to $\zeta$ from $0$ to $\ell$, we get
\begin{equation}\label{eq0310}
\begin{aligned} 
   \|u(t)\|_{L^2(\Omega)}^2
   \leq 
   & \frac{1}{\ell} e^{\frac{\nu\eta_1^{2\alpha}}{2}(\tau+\ell-t)}  \int_{0}^{\ell} \|u(\tau+\zeta)\|_{L^2(\Omega)}^2 d\zeta
    \\
   &
   + \frac{2}{\nu\eta_1^{2\alpha}}
   \frac{1}{\ell}(t-\tau)
   e^{\frac{\nu\eta_1^{2\alpha}}{2}(\tau+\ell-t)}
   \int_{0}^{\ell} 
   \|\theta(\tau+\zeta)\|_{L^2(\Omega)}^2  d\zeta \\
   &
   + \frac{4}{\nu^{3}\eta_1^{6\alpha}} 
   \left(1+\frac{2}{\nu\eta_1^{2\alpha}}\right)
   G_g
   + \frac{2}{\nu\eta_1^{2\alpha}} \left(1+\frac{2}{\nu\eta_1^{2\alpha}}\right)G_f.
\end{aligned}
\end{equation}
Since for all $ t \in \mathbb{R} $, 
\[
(t-\tau)e^{\frac{\nu\eta_1^{2\alpha}}{2}(\tau-t)}
\to 0, \quad\text{as}\quad t - \tau \to \infty, 
\]
it follows that for \cref{eq0310}, there exists a constant $\rho_{1}^{(5)} > 0$ such that for any $B_{\ell} \in \mathcal{D}_{\ell}$, there exists $\tau_{1}^{(5)} = \tau_{1}^{(5)}(B_{\ell}) \geq 0$ such that 
\begin{equation}\label{eq0313}
    \|u(t)\|_{L^2(\Omega)}^2 
    \leq 
    \rho_{1}^{(5)},
\end{equation} 
with
\[
    \rho_{1}^{(5)} 
    := \frac{12}{\nu^{3}\eta_1^{6\alpha}} 
   \left(1+\frac{2}{\nu\eta_1^{2\alpha}}\right)
   G_g
   + \frac{6}{\nu\eta_1^{2\alpha}} \left(1+\frac{2}{\nu\eta_1^{2\alpha}}\right)G_f
\]
for all $t \in \mathbb{R}$ satisfying $t - \tau \geq \tau_{1}^{(5)}$. Integrating $\eqref{eq0304}$ from $t$ to $t+\ell$, we get
\begin{equation}\label{eq0311}
\begin{aligned}
  &\frac{\nu\eta_1^{2\alpha}}{2} \int_{0}^{\ell} \|u(t+\zeta)\|_{L^2(\Omega)}^2 d\zeta + \frac{\nu}{2} \int_{0}^{\ell} \| \Lambda^\alpha u(t+\zeta)\|_{L^2(\Omega)}^2 d\zeta \\
  \leq 
  & 
  \frac{4}{\nu^{2}\eta_1^{4\alpha}} \frac{1}{\ell} e^{\frac{\nu\eta_1^{2\alpha}}{2}(\tau+\ell-t)} \int_{0}^{\ell} \|\theta(\tau+\zeta)\|_{L^2(\Omega)}^2 d\zeta
  + \frac{1}{\ell} e^{\frac{\nu\eta_1^{2\alpha}}{2}(\tau+\ell-t)}  \int_{0}^{\ell} \|u(\tau+\zeta)\|_{L^2(\Omega)}^2 d\zeta \\
  &
  + \frac{2}{\nu\eta_1^{2\alpha}}
   \frac{1}{\ell}(t-\tau)
   e^{\frac{\nu\eta_1^{2\alpha}}{2}(\tau+\ell-t)}
   \int_{0}^{\ell} 
   \|\theta(\tau+\zeta)\|_{L^2(\Omega)}^2  d\zeta \\
  &
  + \frac{2}{\nu^{2}\eta_1^{4\alpha}} 
  \left(1+\frac{2}{\nu\eta_1^{2\alpha}}\right)
  \left(\ell+\frac{2}{\nu\eta_1^{2\alpha}}\right)
  G_g
  + \frac{2}{\nu\eta_1^{2\alpha}} \left(2+[\ell]+\frac{2}{\nu\eta_1^{2\alpha}}\right)G_f.
\end{aligned}
\end{equation}
For \cref{eq0311}, there exists a constant $\rho_{1}^{(6)} > 0$ such that for any $B_{\ell} \in \mathcal{D}_{\ell}$, there exists $\tau_{1}^{(6)} = \tau_{1}^{(6)}(B_{\ell}) \geq 0$ where  
\begin{equation}\label{eq0312}
    \int_{0}^{\ell} \| \Lambda^\alpha u(t+\zeta)\|_{L^2(\Omega)}^2 d\zeta
    \leq 
    \rho_{1}^{(6)},
\end{equation}
with
\[
    \rho_{1}^{(6)} 
    := 
    \frac{16}{\nu\eta_1^{2\alpha}\min\{\nu\eta_1^{2\alpha},\nu\}}
    \left(
    \frac{1}{\nu\eta_1^{2\alpha}} 
    \left(1+\frac{2}{\nu\eta_1^{2\alpha}}\right)
    \left(\ell+\frac{2}{\nu\eta_1^{2\alpha}}\right)
    G_g
    + \left(2+[\ell]+\frac{2}{\nu\eta_1^{2\alpha}}\right)G_f
    \right)
\]
for all $t \in \mathbb{R}$ satisfying $t - \tau \geq \tau_{1}^{(6)}$.
Consequently, for any $\tau \in \mathbb{R}$, there exists a constant 
$$
\rho_1=\rho_{1}^{(1)}+\rho_{1}^{(2)}+\rho_{1}^{(3)}+
\rho_{1}^{(4)}+\rho_{1}^{(5)}+\rho_{1}^{(6)}>0
$$
such that for every $B_{\ell} \in \mathcal{D}_{\ell}$, there exists a time 
$$
\tau_1 =\tau_1(B_{\ell})=\max\{\tau_{1}^{(1)},
            \tau_{1}^{(2)},
            \tau_{1}^{(3)},
            \tau_{1}^{(4)},
            \tau_{1}^{(5)},
            \tau_{1}^{(6)}\}\geq 0 ,
$$
for all strong solutions to \cref{eq0101} with short trajectory $\chi(s,\tau; (u_\tau, \theta_\tau)) \in B_{\ell}$, we obtain
\[
\| u(t) \|_{H_{1}}^{2}
+ \| \theta(t) \|_{H_{2}}^{2}
+ \int_{0}^{\ell} \left(
\| \Lambda^\alpha u(t+\zeta)\|_{L^2(\Omega)}^2 + \| \Lambda^\beta \theta(t+\zeta)\|_{L^2(\Omega)}^2 \right) d\zeta
\leq \rho_1 
\]%
for all $t-\tau \geq \tau_1$. This completes the proof of \cref{lem0301}.
\end{proof}

\begin{lemma}\label{lem0302}
Assume that \cref{ass:forcing} holds. For any $\tau \in \mathbb{R}$, there exists a positive constant $\rho_2$ such that for every bounded subset $B_{\ell} \in \mathcal{D}_{\ell}$, there exists a time $\tau_2 =\tau_2(B_{\ell}) \geq 0 $, for all strong solutions to \cref{eq0101} with short trajectory $\chi(s,\tau; (u_\tau, \theta_\tau)) \in B_{\ell}$, we have 
\[
\| \Lambda^{\alpha} u(t)\|_{L^2(\Omega)}^{2}
+ \int_{0}^{\ell} \| \Lambda^{2\alpha} u(t+\zeta) \|_{L^2(\Omega)}^{2} d\zeta 
\leq \rho_2 
\]
for all $t-\tau \geq \tau_2$.
\end{lemma}
\begin{proof}
Taking the $L^{2}$-inner product of $\eqref{eq0101}_{1}$ with $\Lambda^{2\alpha}u$ and integrating over $\Omega$, we obtain 
\begin{equation}\label{eq0324}
\begin{aligned}
  &\frac{1}{2}\frac{d}{dt}\| \Lambda^{\alpha}u(t) \|_{L^2(\Omega)}^{2}
  +\nu \| \Lambda^{2\alpha} u(t) \|_{L^2(\Omega)}^{2} \\
  = & -\int_{\Omega} \mathbb{P} \left( (u \cdot \nabla) u \right) \cdot \Lambda^{2\alpha} u dx 
  +\int_{\Omega} \mathbb{P} (\theta e_3) \cdot \Lambda^{2\alpha} u dx
  +\int_{\Omega} \mathbb{P} f \cdot \Lambda^{2\alpha} u dx.
\end{aligned}
\end{equation}
In the first term on the right-hand side of \cref{eq0324}, since $\alpha>\frac{5}{4}$, Gagliardo-Nirenberg inequalities hold. We thus obtain
\begin{equation}\label{eq0325}
\begin{aligned}
  \int_{\Omega} \mathbb{P} \left( (u \cdot \nabla) u \right) \cdot \Lambda^{2\alpha} u dx
  & \leq \int_{\Omega} |u|
   |\nabla u|
   |\Lambda^{2\alpha} u| dx 
   \\
  & \leq \|u(t)\|_{L^\infty(\Omega)}
  \|\nabla u(t)\|_{L^2(\Omega)}
  \|\Lambda^{2\alpha} u(t)\|_{L^2(\Omega)}
  \\
  & \leq C \|u(t)\|_{L^2(\Omega)}^{\frac{8\alpha-5}{4\alpha}}
  \|\Lambda^{2\alpha} u(t)\|_{L^2(\Omega)}^{\frac{4\alpha+5}{4\alpha}}
  \\
  & \leq \frac{\nu}{6} \|\Lambda^{2\alpha} u(t)\|_{L^2(\Omega)}^{2}+ C
  \|u(t)\|_{L^2(\Omega)}^{\frac{2(8\alpha-5)}{4\alpha-5}}.
\end{aligned}
\end{equation}
By Young's inequality, one has
\begin{equation}\label{eq0326}
\begin{aligned}
  \int_{\Omega} \mathbb{P} (\theta e_3) \cdot \Lambda^{2\alpha} u dx \leq \frac{\nu}{6} \|\Lambda^{2\alpha} u(t)\|_{L^2(\Omega)}^{2}+ \frac{3}{2\nu}
  \|\theta(t)\|_{L^2(\Omega)}^{2}.
\end{aligned}
\end{equation}
Similarly, we have
\begin{equation}\label{eq0327}
\begin{aligned}
  \int_{\Omega} \mathbb{P} f \cdot \Lambda^{2\alpha} u dx\leq \frac{\nu}{6} \|\Lambda^{2\alpha} u(t)\|_{L^2(\Omega)}^{2}+ \frac{3}{2\nu}
  \|f(t)\|_{L^2(\Omega)}^{2}.
\end{aligned}
\end{equation}
It follows from \cref{eq0325,eq0326,eq0327} that
\begin{equation}\label{eq0328}
  \frac{d}{dt}\| \Lambda^{\alpha}u(t) \|_{L^2(\Omega)}^{2}
  +\nu \| \Lambda^{2\alpha} u(t) \|_{L^2(\Omega)}^{2}
  \leq  C
  \|u(t)\|_{L^2(\Omega)}^{\frac{2(8\alpha-5)}{4\alpha-5}}
  +\frac{3}{\nu}
  \|\theta(t)\|_{L^2(\Omega)}^{2}
  +\frac{3}{\nu}
  \|f(t)\|_{L^2(\Omega)}^{2}.
\end{equation}
By computations analogous to those in \cref{eq0314,eq0312}, integrating \cref{eq0304} over $(t-\ell,t)$ yields the following: for $\alpha > \frac{5}{4}$, there exists a positive constant $\rho_{1}^{(7)}$ and a time $\tau_{2}^{(1)} \geq \tau_{1}$ such that for all $t-\tau \geq \tau_{2}^{(1)}$,
\begin{align*}
  \int_{t-\ell}^{t} \|\Lambda^\alpha u(s)\|_{L^2(\Omega)}^2 ds
  =\int_{0}^{\ell} \|\Lambda^\alpha u(t-\zeta)\|_{L^2(\Omega)}^2 d\zeta
  \leq \rho_{1}^{(7)}.
\end{align*}
Then, for any \(\zeta \in (0, \ell)\), integrating \cref{eq0328} over \([t - \zeta, t]\) and integrating the result over \((0, \ell)\) with respect to \(\zeta\), we obtain
\begin{equation}\label{eq0329}
\begin{aligned}
  \ell \| \Lambda^{\alpha} u(t)\|_{L^2(\Omega)}^{2}
  \leq & C\int_{0}^{\ell}\int_{t-\zeta}^{t}
  \|u(s)\|_{L^2(\Omega)}^{\frac{2(8\alpha-5)}{4\alpha-5}}ds d\zeta
  +\frac{3}{\nu}\int_{0}^{\ell}\int_{t-\zeta}^{t}
  \|\theta(s)\|_{L^2(\Omega)}^{2}
   ds d\zeta
  \\
  &+\frac{3}{\nu}\int_{0}^{\ell}\int_{t-\zeta}^{t}
  \|f(s)\|_{L^2(\Omega)}^{2} ds d\zeta
  +\int_{0}^{\ell} \| \Lambda^{\alpha} u(t-\zeta)\|_{L^2(\Omega)}^{2}d\zeta
  \\
  \leq & C \ell \int_{t-\ell}^{t}
  \|u(s)\|_{L^2(\Omega)}^{\frac{2(8\alpha-5)}{4\alpha-5}}
   ds 
  +\frac{3}{\nu}\ell \int_{t-\ell}^{t}
  \|\theta(s)\|_{L^2(\Omega)}^{2} ds 
  \\
  &+\frac{3}{\nu} \ell \int_{t-\ell}^{t}
  \|f(s)\|_{L^2(\Omega)}^{2} ds
  +\int_{0}^{\ell} \|\Lambda^{\alpha}  u(t-\zeta)\|_{L^2(\Omega)}^{2} d\zeta
  \\
  \leq & C \ell^{2}
  \rho_{1}^{\frac{8\alpha-5}{4\alpha-5}}
  +\frac{3}{\nu} \ell^2 \rho_{1}
  +\frac{3}{\nu} \ell (1+[\ell]) G_{f}
  +\rho_{1}^{(7)}.
\end{aligned}
\end{equation}
for all $t-\tau \geq \tau_{2}^{(1)}$. Now we obtain that
\begin{equation}\label{eq0330}
\| \Lambda^{\alpha} u(t)\|_{L^2(\Omega)}^{2}
  \leq  \rho_{2}^{(1)},
\end{equation}
where
\begin{equation}
    \rho_{2}^{(1)}= C \ell
  \rho_{1}^{\frac{8\alpha-5}{4\alpha-5}}
  +\frac{3}{\nu} \ell \rho_{1}
  +\frac{3}{\nu}  (1+[\ell]) G_{f}
  +\frac{1}{\ell}\rho_{1}^{(7)}
\end{equation}
for all $t \in \mathbb{R}$ when $t - \tau \geq \tau_{2}^{(1)}$.
Integrating both sides of $\eqref{eq0328}$ from $t$ to $t+\ell$ yields
\begin{equation}
\begin{aligned}
  \nu \int_{0}^{\ell} \| \Lambda^{2\alpha} u(t+\zeta) \|_{L^2(\Omega)}^{2} d\zeta 
  \leq & C\int_{0}^{\ell}\|u(t+\zeta)\|_{L^2(\Omega)}^{\frac{2(8\alpha-5)}{4\alpha-5}}
    d\zeta
  +\frac{3}{\nu}\int_{t}^{t+\ell} \|\theta(s)\|_{L^2(\Omega)}^{2} ds \\
  &+\frac{3}{\nu}\int_{t}^{t+\ell} \|f(s)\|_{L^2(\Omega)}^{2} ds +\|\Lambda^{\alpha} u(t)\|_{L^2(\Omega)}^{2}\\
  \leq & 
   C\ell \rho_{1}^{\frac{8\alpha-5}{4\alpha-5}}
  +\frac{3}{\nu} \ell \rho_{1}
  +\frac{3}{\nu}(1+[\ell]) G_{f}
  +\rho_{2}^{(1)}.
\end{aligned}
\end{equation}
Hence, we derive that
\begin{equation}
\int_{0}^{\ell} \| \Lambda^{2\alpha} u(t+\zeta) \|_{L^2(\Omega)}^{2} d\zeta 
  \leq \rho_{2}^{(2)},
\end{equation}
$$
\rho_{2}^{(2)}=\frac{1}{\nu}\left(C \ell \rho_{1}^{\frac{8\alpha-5}{4\alpha-5}}
  +\frac{3}{\nu} \ell \rho_{1}
  +\frac{3}{\nu}(1+[\ell]) G_{f}
  +\rho_{2}^{(1)}\right).
$$
Consequently, for any $\tau \in \mathbb{R}$, there exists a constant 
$$
\rho_2:=\rho_{2}^{(1)}+\rho_{2}^{(2)}>0
$$
such that for every $B_{\ell} \in \mathcal{D}_{\ell}$, there exists a time 
$$
\tau_2 =\tau_2(B_{\ell})=\tau_{2}^{(1)} \geq \tau_{1}\geq 0 ,
$$
for all solutions to \cref{eq0101} with short trajectory $\chi(s,\tau; (u_\tau, \theta_\tau)) \in B_{\ell}$, we obtain
\[
\| \Lambda^{\alpha} u(t)\|_{L^2(\Omega)}^{2}
+ \int_{0}^{\ell} \| \Lambda^{2\alpha} u(t+\zeta) \|_{L^2(\Omega)}^{2} d\zeta 
\leq \rho_2 
\]
for all $t-\tau \geq \tau_2$. Thus we complete the proof of \cref{lem0302}.
\end{proof}
 
\begin{lemma}\label{lem0303}
Assume that \cref{ass:forcing} holds. For any $\tau \in \mathbb{R}$, there exists a positive constant $\rho_3$ such that for every bounded subset $B_{\ell} \in \mathcal{D}_{\ell}$, there exists a time $\tau_3 =\tau_3(B_{\ell}) \geq 0 $, for all strong solutions to \cref{eq0101} with short trajectory $\chi(s,\tau; (u_\tau, \theta_\tau)) \in B_{\ell}$, we have 
\[
\| \Lambda^{\beta} \theta(t)\|_{L^2(\Omega)}^{2}
+ \int_{0}^{\ell} \| \Lambda^{2\beta} \theta(t+\zeta) \|_{L^2(\Omega)}^{2} d\zeta 
\leq \rho_3
\]
for all $t-\tau \geq \tau_3$.
\end{lemma}
\begin{proof}
Multiplying both sides of $\eqref{eq0101}_{2}$ by $\Lambda^{2\beta}\theta$ and integrating over $\Omega$, we derive
\begin{align}\label{eq030301}
  \frac{1}{2}\frac{d}{dt}\|\Lambda^{\beta}\theta(t)\|_{L^2(\Omega)}^2 + \kappa\| \Lambda^{2\beta} \theta(t)\|_{L^2(\Omega)}^2
  = - \int_{\Omega}\left((u\cdot\nabla)\theta\right)\Lambda^{2\beta}\theta dx
  + \int_{\Omega}g\Lambda^{2\beta}\theta dx.
\end{align}
Using $\beta>\frac{5}{8}$, Gagliardo-Nirenberg inequality and Young's inequality, we find
\begin{align*}
   -\int_{\Omega}\left((u\cdot\nabla)\theta\right)\Lambda^{2\beta}\theta dx
   & \leq \int_{\Omega}|u||\nabla \theta||\Lambda^{2\beta}\theta|dx \\
   & \leq
   \|u(t)\|_{L^{12}(\Omega)}
   \|\nabla \theta(t)\|_{L^{\frac{12}{5}}(\Omega)}
   \|\Lambda^{2\beta}\theta(t)\|_{L^2(\Omega)} \\
   & \leq C
   \|\Lambda^{\frac{5}{4}}u(t)\|_{L^{2}(\Omega)}
   \|\theta(t)\|_{L^{2}(\Omega)}^{\frac{8\beta-5}{8\beta}}
   \|\Lambda^{2\beta}\theta(t)\|_{L^2(\Omega)}^{\frac{8\beta+5}{8\beta}} \\
   & \leq \frac{\kappa}{4}
   \|\Lambda^{2\beta}\theta(t)\|_{L^2(\Omega)}^{2} 
   +C
   \|\Lambda^{\frac{5}{4}}u(t)\|_{L^{2}(\Omega)}^{\frac{16\beta}{8\beta-5}}
   \|\theta(t)\|_{L^{2}(\Omega)}^{2}.
\end{align*}
It follows from Young's inequality that
\begin{align*}
\int_{\Omega}g\Lambda^{2\beta}\theta dx \leq \frac{\kappa}{4}\|\Lambda^{2\beta}\theta(t)\|_{L^2(\Omega)}^{2} 
+\frac{1}{\kappa}\|g(t)\|_{L^2(\Omega)}^{2}.
\end{align*}
Substituting the above estimates into \cref{eq030301}, there holds
\begin{align}\label{eq030302}
  \frac{d}{dt}\|\Lambda^{\beta}\theta(t)\|_{L^2(\Omega)}^2 + \kappa\| \Lambda^{2\beta} \theta(t)\|_{L^2(\Omega)}^2
  \leq C
   \|\Lambda^{\frac{5}{4}}u(t)\|_{L^{2}(\Omega)}^{\frac{16\beta}{8\beta-5}}
   \|\theta(t)\|_{L^{2}(\Omega)}^{2}
   +\frac{2}{\kappa}\|g(t)\|_{L^2(\Omega)}^{2}.
\end{align}
For any $\zeta \in (0, \ell)$, integrating \cref{eq030302} over the interval $[t - \zeta, t]$ and then integrating the result with respect to $\zeta$ over $(0, \ell)$ yields
\begin{equation}\label{eq030303}
\begin{aligned}
  \ell\|\Lambda^{\beta}\theta(t)\|_{L^2(\Omega)}^2 
  \leq&  C
   \int_{0}^{\ell}\int_{t-\zeta}^{t}
   \|\Lambda^{\frac{5}{4}}u(s)\|_{L^{2}(\Omega)}^{\frac{16\beta}{8\beta-5}}
   \|\theta(s)\|_{L^{2}(\Omega)}^{2} ds d\zeta
   \\
   &+\frac{2}{\kappa}\int_{0}^{\ell}\int_{t-\zeta}^{t}\|g(s)\|_{L^2(\Omega)}^{2}  ds d\zeta  +\int_{0}^{\ell}\|\Lambda^{\beta}\theta(t-\zeta)\|_{L^2(\Omega)}^{2}   d\zeta \\
   \leq&  C
   \ell\int_{t-\ell}^{t}
   \|\Lambda^{\frac{5}{4}}u(s)\|_{L^{2}(\Omega)}^{\frac{16\beta}{8\beta-5}}
   \|\theta(s)\|_{L^{2}(\Omega)}^{2} ds
   \\
   &+\frac{2}{\kappa}\ell\int_{t-\ell}^{t}\|g(s)\|_{L^2(\Omega)}^{2}  ds +\int_{0}^{\ell}\|\Lambda^{\beta}\theta(t-\zeta)\|_{L^2(\Omega)}^{2}   d\zeta \\
   \leq&  C
   \ell\int_{t-\ell}^{t}
   \|\Lambda^{\frac{5}{4}}u(s)\|_{L^{2}(\Omega)}^{\frac{16\beta}{8\beta-5}}
   \|\theta(s)\|_{L^{2}(\Omega)}^{2} ds
   \\
   &+\frac{2}{\kappa}\ell(1+[\ell])G_{g} +\int_{0}^{\ell}\|\Lambda^{\beta}\theta(t-\zeta)\|_{L^2(\Omega)}^{2}   d\zeta.
\end{aligned}
\end{equation}
Combining \cref{eq0317}, and then integrating \cref{eq0301} from $t-\ell$ to $t$, one can derive that there exists a constant $\rho_{1}^{(8)}>0$ and $\tau_{3}^{(1)} \geq \tau_{2}$ such that for all $t-\tau \geq \tau_{3}^{(1)}$,
\begin{align}\label{eq030304}
  \int_{t-\ell}^{t} \|\Lambda^\beta \theta(s)\|_{L^2(\Omega)}^2 ds
  =\int_{0}^{\ell} \|\Lambda^\beta \theta(t-\zeta)\|_{L^2(\Omega)}^2 d\zeta
  \leq \rho_{1}^{(8)}.
\end{align}
Combining \cref{lem0301,lem0302,eq030304}, we obtain that for all $t-\tau \geq \tau_{3}^{(1)}$,
\begin{align*}
  \ell\|\Lambda^{\beta}\theta(t)\|_{L^2(\Omega)}^2 
  \leq&C\ell^{2}\rho_{1}\rho_{2}^{\frac{8\beta}{8\beta-5}}
  +\frac{2}{\kappa}\ell(1+[\ell])G_{g} +\rho_{1}^{(8)}.
\end{align*}
Hence, we have 
\begin{align}\label{eq030305}
  \|\Lambda^{\beta}\theta(t)\|_{L^2(\Omega)}^2 
  \leq \rho_{3}^{(1)},
\end{align}
where
\begin{align*}
  \rho_{3}^{(1)}
  =
  \frac{1}{\ell}
  \left( C\ell^{2}\rho_{1}\rho_{2}^{\frac{8\beta}{8\beta-5}}
  +\frac{2}{\kappa}\ell(1+[\ell])G_{g} +\rho_{1}^{(8)}
  \right),
\end{align*}
for all $t-\tau \geq \tau_{3}^{(1)}$.
Integrating \cref{eq030302} from $t$ to $t+\ell$ yields
\begin{align*}
  & \kappa\int_{0}^{\ell}\|\Lambda^{2\beta} \theta(t+\zeta)\|_{L^2(\Omega)}^2 d\zeta \\
  \leq & C
  \int_{t}^{t+\ell}
   \|\Lambda^{\frac{5}{4}}u(s)\|_{L^{2}(\Omega)}^{\frac{16\beta}{8\beta-5}}
   \|\theta(s)\|_{L^{2}(\Omega)}^{2} ds
   +\frac{2}{\kappa}
   \int_{t}^{t+\ell}\|g(s)\|_{L^2(\Omega)}^{2} ds
   +\|\Lambda^{\beta}\theta(t)\|_{L^2(\Omega)}^2 \\
  \leq & C \ell
  \rho_{1}\rho_{2}^{\frac{8\beta}{8\beta-5}}
   +\frac{2}{\kappa}(1+[\ell])G_{g} 
   +\rho_{3}^{(1)}.
\end{align*}
Hence by a direct computation, we have
\begin{align}\label{eq030306}
 \int_{0}^{\ell}\|\Lambda^{2\beta} \theta(t+\zeta)\|_{L^2(\Omega)}^2 d\zeta \leq \rho_{3}^{(2)},
\end{align}
where
\begin{align*}
   \rho_{3}^{(2)} = \frac{1}{\kappa} \left( C \ell
  \rho_{1}\rho_{2}^{\frac{8\beta}{8\beta-5}}
   +\frac{2}{\kappa}(1+[\ell])G_{g} 
   +\rho_{3}^{(1)} \right).
\end{align*}
For any $\tau \in \mathbb{R}$, there exists a constant 
$$
\rho_3=\rho_{3}^{(1)}+\rho_{3}^{(2)}>0
$$
such that for all $B_{\ell} \in \mathcal{D}_{\ell}$, there exists a time 
$$
\tau_3 =\tau_3(B_{\ell})=\tau_{3}^{(1)} \geq \tau_{2}\geq 0 ,
$$
at which solutions to equation \cref{eq0101} with a short trajectory $\chi(s,\tau; (u_\tau, \theta_\tau)) \in B_{\ell}$ satisfies 
\[
\| \Lambda^{\beta} \theta(t)\|_{L^2(\Omega)}^{2}
+ \int_{0}^{\ell} \| \Lambda^{2\beta} \theta(t+\zeta) \|_{L^2(\Omega)}^{2} d\zeta 
\leq \rho_3
\]
for all $t-\tau \geq \tau_3$. This concludes the proof of \cref{lem0303}.
\end{proof}

\begin{lemma}\label{lem0304}
Assume that \cref{ass:forcing} holds. For any $\tau \in \mathbb{R}$, there exists a constant $\rho_{4}>0$ such that for every bounded subset $B_{\ell} \in \mathcal{D}_{\ell}$, there exists a time $\tau_4 =\tau_4(B_{\ell}) \geq 0 $, for all strong solutions to \cref{eq0101} with short trajectory $\chi(s,\tau; (u_\tau, \theta_\tau)) \in B_{\ell}$, we have 
\[
\| \Lambda^{\alpha} u(t)\|_{L^2(\Omega)}^{2}
+ \| \Lambda^{\beta} \theta(t)\|_{L^2(\Omega)}^{2}
+ \int_{0}^{\ell}\left( \| \Lambda^{2\alpha} u(t+\zeta) \|_{L^2(\Omega)}^{2} + \| \Lambda^{2\beta} \theta(t+\zeta) \|_{L^2(\Omega)}^{2} \right)  d\zeta 
\leq \rho_{4}
\]
for all $t-\tau \geq \tau_{4}$.
\end{lemma}
\begin{proof}
Combining \cref{lem0302,lem0303}, we find that for $\rho_{4}:=\rho_{2}+\rho_{3}$, there exists $\tau_{4} \geq \tau_{3}$ such that \cref{lem0304} holds.
\end{proof}

From \cref{lem0304}, we deduce the following \cref{cor01}.

\begin{corollary}\label{cor01}
Assume that \cref{ass:forcing} holds. For any $\tau \in \mathbb{R}$, let $\rho_{4}>0$ be the constant from \cref{lem0304}. For every bounded subset $B \in \mathcal{D}$, there exists a time $\tau_4' =\tau_4'(B) \geq 0 $, for all strong solutions to \cref{eq0101} with $ (u_\tau, \theta_\tau) \in B$, we have 
\[
\| \Lambda^{\alpha} u(t)\|_{L^2(\Omega)}^{2}
+ \| \Lambda^{\beta} \theta(t)\|_{L^2(\Omega)}^{2}
+ \int_{0}^{\ell}\left( \| \Lambda^{2\alpha} u(t+\zeta) \|_{L^2(\Omega)}^{2} + \| \Lambda^{2\beta} \theta(t+\zeta) \|_{L^2(\Omega)}^{2} \right)  d\zeta 
\leq \rho_{4}
\]
for all $t-\tau \geq \tau_4'$.
\end{corollary}

For $\alpha>\frac{5}{4}$ and $\beta>\frac{5}{8}$, define the set
\[
B_0 = \left\{ (u, \theta) \in H_\sigma^{\alpha}(\Omega) \times H^\beta(\Omega) : \|u\|_{H^{\alpha}(\Omega)}^2 + \|\theta\|_{H^\beta(\Omega)}^2 \leq \rho_4 \right\}.
\]
Then \( B_0 \in \mathcal{D} \). For all $t \in \mathbb{R}$, by \cref{cor01}, there exists \( \tau_0 = \tau_0(B_0) \geq 0 \) such that for all \( t - \tau \geq \tau_0 \) and any initial data \( (u_\tau, \theta_\tau) \in B_0 \),  
\[
U(t, \tau) B_0 \subset B_0.
\]
For any $t\in \mathbb{R}$, define
\[
B_1(t) = \overline{\bigcup_{\substack{\tau \in [t - \tau_0, t],\ s \in [t - \tau_0, t]}} \left\{ U(s, \tau)(u_\tau, \theta_\tau) : \forall\ (u_\tau, \theta_\tau) \in B_0 \right\}}^{H_\sigma^{\alpha}(\Omega) \times H^{\beta}(\Omega)},
\]
and
\[
\mathcal{B}_0^\ell(t) = \left\{ \chi \in X_\ell : e_0(\chi) \in B_1(t) \right\}.
\]
Specifically, for $\tau = t$, we have 
$$
B_1(\tau) = \left\{ (u_\tau, \theta_\tau) : \forall\ (u_\tau, \theta_\tau) \in B_0 \right\},\quad
\mathcal{B}_0^\ell(\tau) = \left\{ \chi \in X_\ell : e_0(\chi) \in B_1(\tau) \right\}.
$$
It follows from \cref{cor01} that
\[
U(t, \tau) B_1(\tau) \subset B_1(t),
\]
and
\[
L(t, \tau) \mathcal{B}_0^\ell(\tau) \subset \mathcal{B}_0^\ell(t),
\]
for any $t \in \mathbb{R}$ with $\tau \leq t$, and $B_1(t) \in \mathcal{D}$.

\begin{lemma}\label{lem0306}
Assume that \cref{ass:forcing} holds. For any $\tau \in \mathbb{R}$, there exists a constant $\rho_5 >0$ such that for every bounded subset $B_{\ell} \in \mathcal{D}_{\ell}$, there exists a time $\tau_5 =\tau_5(B_{\ell}) \geq 0 $, for all strong solutions to \cref{eq0101} with short trajectory $\chi(s,\tau; (u_\tau, \theta_\tau)) \in B_{\ell}$, we have 
$$
\int_{0}^{\ell}\left( 
\|u_t(t+\zeta) \|_{L^{2}(\Omega)} ^{2}
+ 
\|\theta_t(t+\zeta) \|_{L^{2}(\Omega)}^{2} \right)  d\zeta 
\leq \rho_{5}
$$
for all $t-\tau \geq \tau_{5}$.
\end{lemma}
\begin{proof}
We obtain by \cref{eq0201}
\begin{equation}
\begin{cases}
u_t = - \nu \Lambda^{2\alpha} u - \mathbb{P} \left( (u \cdot \nabla) u \right) + \mathbb{P} (\theta e_3) + \mathbb{P} f, \\
\theta_t = - \kappa \Lambda^{2\beta} \theta - (u \cdot \nabla) \theta + g.
\end{cases}
\end{equation}
By H\"{o}lder’s inequality, $\alpha>\frac{5}{4}$, and $\beta>\frac{5}{8}$, together with the definition of dual norm, we derive that 
\begin{equation}\label{eq030703}
\begin{aligned}
\|u_t\|_{L^{2}(\Omega)}
\leq &
\nu \| \Lambda^{2\alpha} u \|_{L^2(\Omega)} 
+ \|(u \cdot \nabla) u \|_{L^{2}(\Omega)} 
+\| \theta \|_{L^{2}(\Omega)} 
+\| f \|_{L^{2}(\Omega)} \\
\leq &
\nu \| \Lambda^{2\alpha} u \|_{L^2(\Omega)} 
+\|u \|_{L^{12}(\Omega)} 
\| \nabla u \|_{L^{\frac{12}{5}}(\Omega)} 
+\| \theta \|_{L^{2}(\Omega)} 
+\| f \|_{L^{2}(\Omega)} \\
\leq &
\nu \| \Lambda^{2\alpha} u \|_{L^2(\Omega)} 
+C\| \Lambda^{\alpha} u \|_{L^{2}(\Omega)}^{2}
+\| \theta \|_{L^{2}(\Omega)} 
+\| f \|_{L^{2}(\Omega)},
\end{aligned}
\end{equation}
and
\begin{equation}\label{eq030704}
\begin{aligned}
\| \theta_t \|_{L^2(\Omega)}
\leq &
\kappa \| \Lambda^{2\beta} \theta  \|_{L^2(\Omega)}
+ \|(u \cdot \nabla) \theta \|_{L^{2}(\Omega)}
+ \| g \|_{L^{2}(\Omega)}\\
\leq &
\kappa \| \Lambda^{2\beta} \theta  \|_{L^2(\Omega)}
+ \|u \|_{L^{12}(\Omega)} \| \nabla \theta \|_{L^{\frac{12}{5}}(\Omega)} 
+ \| g \|_{L^{2}(\Omega)}\\
\leq &
\kappa \| \Lambda^{2\beta} \theta  \|_{L^2(\Omega)}
+C \| \Lambda^{\alpha} u \|_{L^{2}(\Omega)}
\|\Lambda^{2\beta} \theta \|_{L^{2}(\Omega)}
+ \| g \|_{L^{2}(\Omega)}.
\end{aligned}
\end{equation}
Hence, we can derive that
\begin{equation}\label{eq030705}
\begin{aligned}
\|u_t\|_{L^{2}(t,t+\ell;L^{2}(\Omega))}
\leq &
\nu \| u \|_{L^{2}(t,t+\ell;H_\sigma^{2\alpha}(\Omega))} 
+C \|u \|_{L^{\infty}(t,t+\ell;H_\sigma^{\alpha}(\Omega))} 
\| u \|_{L^{2}(t,t+\ell;H_\sigma^{\alpha}(\Omega))}\\ 
&+ \| \theta \|_{L^{2}(t,t+\ell;L^{2}(\Omega))} 
+ \| f \|_{L^{2}(t,t+\ell;L^{2}(\Omega))} \\
\leq &
\nu \| u \|_{L^{2}(t,t+\ell;H_\sigma^{2\alpha}(\Omega))} 
+C \|u \|_{L^{\infty}(t,t+\ell;H_\sigma^{\alpha}(\Omega))} 
\| u \|_{L^{2}(t,t+\ell;H_\sigma^{\alpha}(\Omega))}\\ 
&+ C\| \theta \|_{L^{\infty}(t,t+\ell;L^{2}(\Omega))} 
+ \| f \|_{L^{2}(t,t+\ell;L^{2}(\Omega))} ,
\end{aligned}
\end{equation}
and
\begin{equation}\label{eq030706}
\begin{aligned}
\| \theta_t \|_{L^{2}(t,t+\ell;L^{2}(\Omega))}
\leq&
\kappa \| \theta \|_{L^{2}(t,t+\ell;H^{2\beta}(\Omega))}
+ C
\| u \|_{L^{\infty}(t,t+\ell;H_\sigma^{\alpha}(\Omega))}
\|\theta \|_{L^{2}(t,t+\ell;H^{2\beta}(\Omega))}
\\
&+ \| g \|_{L^{2}(t,t+\ell;L^{2}(\Omega))}.
\end{aligned}
\end{equation}
Combining this with Lemmas \cref{lem0301,lem0304}, we infer that estimates \cref{eq030705,eq030706} are bounded.
Consequently, there exists $\rho_5 >0$ such that 
\begin{equation}\label{eq030707}
\int_{0}^{\ell}\left( 
\|u_t(t+\zeta) \|_{L^{2}(\Omega)}^{2} 
+ 
\|\theta_t(t+\zeta) \|_{L^{2}(\Omega)}^{2} \right)  d\zeta 
\leq \rho_{5}
\end{equation}
for all $t-\tau \geq \tau_{5}=\tau_{4}$.
\end{proof}

Define the set
\begin{equation}\label{defWl}
W_{\ell}:=
\left\{
 \chi\in X_{\ell}: 
 \chi\in {L^{2}(\tau,\tau+\ell;H_\sigma^{2\alpha}(\Omega) \times H^{2\beta}(\Omega))};
 \chi_{t}\in {L^{2}(\tau,\tau+\ell;H_{1} \times H_{2})}
\right\}
\end{equation}
equipped with the norm 
\begin{equation*}
\|\chi\|_{W_{\ell}}
=\|\chi\|_{L^{2}(\tau,\tau+\ell;H_\sigma^{2\alpha}(\Omega) \times H^{2\beta}(\Omega))}
+\|\chi_{t}\|_{L^{2}(\tau,\tau+\ell;H_{1} \times H_{2})}.
\end{equation*}
Next we define $\hat{\mathcal{B}}_{1}^{\ell}(t):=\left\{ \mathcal{B}_{1}^{\ell}(t):t\in\mathbb{R} \right\}$, where 
\[
\mathcal{B}_{1}^{\ell}(t):=\left\{\chi \in X_{\ell}:\|\chi\|_{L^{2}(t,t+\ell;H_\sigma^{2\alpha}(\Omega) \times H^{2\beta}(\Omega))}
+\|\chi_{t}\|_{L^{2}(\tau,\tau+\ell;H_{1} \times H_{2})}\leq \rho_{6} \right\},
\]
where $\rho_{6}:=2+\rho_{4}+\rho_{5}$.

Applying \cref{lem0304,lem0306} and the above expression, we can derive that
\[
L(t, \tau) \mathcal{B}_0^\ell(\tau) \subset \mathcal{B}_{1}^\ell(t),
\]
for all $t-\tau \geq \tau_{5}$.

\begin{lemma}\label{lem0307}
Assume that \cref{ass:forcing} holds. Then for any $t \geq \tau$, the following closure property holds:
\[
\overline{L(t, \tau) \mathcal{B}_0^\ell(\tau)}^{L^{2}(\tau,\tau+\ell;H_\sigma^{\alpha}(\Omega) \times H^{\beta}(\Omega))} \subset \mathcal{B}_0^\ell(t).
\]
\end{lemma}

\begin{proof}
Since $L(t, \tau) \mathcal{B}_0^\ell(\tau) \subset \mathcal{B}_0^\ell(t)$ is already established for any $t\geq\tau$, it suffices to show that for any fixed $t \in \mathbb{R}$, 
\[
\overline{\mathcal{B}_0^\ell(t)}^{L^{2}(\tau,\tau+\ell;H_\sigma^{\alpha}(\Omega) \times H^{\beta}(\Omega))} \subset \mathcal{B}_0^\ell(t).
\]  
Fix $\tau \in \mathbb{R}$ and let $$\chi_0 \in \overline{\mathcal{B}_0^\ell(\tau)}^{L^{2}(\tau,\tau+\ell;H_\sigma^{\alpha}(\Omega) \times H^{\beta}(\Omega))}.$$
By definition, there exists a sequence of $\ell$-trajectories $\{\chi_n\}_{n=1}^{\infty} \subset \mathcal{B}_0^\ell(\tau)$ such that:
\[
\chi_n \to \chi_0 \quad \text{in} \quad L^{2}(\tau, \tau+\ell; H_\sigma^{\alpha}(\Omega) \times H^{\beta}(\Omega)) \quad \text{as} \quad n \to \infty.
\]
From the definition of $\mathcal{B}_0^\ell(\tau)$, we have $e_0(\chi_n) \in B_1(\tau)$ for all $n \in \mathbb{N}$. Since $B_1(\tau)$ is a bounded subset of $H_\sigma^{\alpha}(\Omega) \times H^{\beta}(\Omega)$, there exists a subsequence $\{\chi_{n_j}\}_{j=1}^{\infty}$ and $(u_{\tau}, \theta_{\tau}) \in H_\sigma^{\alpha}(\Omega) \times H^{\beta}(\Omega)$ such that:
\[
e_0(\chi_{n_j}) \rightharpoonup (u_{\tau}, \theta_{\tau}) \quad \text{weakly in} \quad H_\sigma^{\alpha}(\Omega) \times H^{\beta}(\Omega) \quad \text{as} \quad j \to \infty.
\]
Since $B_1(\tau)$ is convex and closed, it is also weakly closed. Therefore, $(u_{\tau}, \theta_{\tau}) \in B_1(\tau)$. Now, these estimates established in \cref{lem0304,lem0306} imply that the sequence $\{\chi_{n_j}\}$ is bounded. Therefore, we can extract a further subsequence (still denoted by $\{\chi_{n_j}\}$) that converges strongly in $L^{2}(\tau, \tau+\ell; H_\sigma^{\alpha}(\Omega) \times H^{\beta}(\Omega))$ to some limit. Since we already know that $\chi_{n_j} \to \chi_0$ in this space, it follows that the limit is $\chi_0$. Now, by Theorem \ref{thm0203} on well-posedness and continuous dependence on initial data, and since $e_0(\chi_{n_j}) \to (u_{\tau}, \theta_{\tau})$ weakly in $H_\sigma^{\alpha}(\Omega) \times H^{\beta}(\Omega)$, it follows that $\chi_0$ is the unique strong solution of \eqref{eq0101} with initial data $(u_{\tau}, \theta_{\tau})$ at time $\tau$. Finally, since $(u_{\tau}, \theta_{\tau}) \in B_1(\tau)$ and $\chi_0$ is the corresponding solution, we have $\chi_0 \in \mathcal{B}_0^\ell(\tau)$ by definition.
\end{proof}

\begin{lemma}\label{lem0308}
Assume that \cref{ass:forcing} holds. Then, for any $\tau \in \mathbb{R}$, the mapping $L(t, \tau) : X_\ell \to X_\ell$ is Lipschitz continuous on $\mathcal{B}_0^\ell(\tau) \subset X_\ell$ for all $t \geq \tau + \ell$.  
\end{lemma}
\begin{proof}
For all fixed $\tau \in \mathbb{R}$ and any $\chi_{1}, \chi_{2} \in \mathcal{B}_{0}^{\ell}(\tau)$, let $L(t, \tau)\chi_1 = (u_1(t), \theta_1(t))$ and $L(t, \tau)\chi_2 = (u_2(t), \theta_2(t))$ for any fixed $t \geq \tau + \ell$. Let $(U, \Theta) = (u_1 - u_2, \theta_1 - \theta_2)$. Then we have
\begin{equation}\label{eq0501}
\begin{cases}
\partial_t U + \nu \Lambda^{2\alpha} U + \mathbb{P} \left( (u_1 \cdot \nabla) U \right) + \mathbb{P} \left( (U \cdot \nabla) u_2 \right) = \mathbb{P} (\Theta e_3), \\
\partial_t \Theta + \kappa \Lambda^{2\beta} \Theta + (u_1 \cdot \nabla) \Theta + (U \cdot \nabla) \theta_2 = 0.
\end{cases}
\end{equation}
Multiplying \eqref{eq0501} by $(\Lambda^{2\alpha}U,\Lambda^{2\beta}\Theta)$, and integrating over $\Omega$ yields
\begin{equation}\label{eq0502}
\begin{aligned}
& \frac{1}{2}\frac{d}{dt}
\left( 
\| \Lambda^{\alpha}U(t) \|_{L^2(\Omega)}^{2} + \| \Lambda^{\beta}\Theta (t) \|_{L^2(\Omega)}^{2} 
\right)
+\nu \| \Lambda^{2\alpha}U(t) \|_{L^2(\Omega)}^{2} 
+\kappa \| \Lambda^{2\beta}\Theta(t) \|_{L^2(\Omega)}^{2} 
\\
=
& \int_{\Omega} \mathbb{P} (\Theta e_3) \cdot \Lambda^{2\alpha}U dx
- \int_{\Omega} \mathbb{P} \left( (u_1 \cdot \nabla) U \right) \cdot \Lambda^{2\alpha}U dx
- \int_{\Omega} \mathbb{P} \left( (U \cdot \nabla) u_2 \right) \cdot \Lambda^{2\alpha}U dx\\
&
- \int_{\Omega} \left( (u_1 \cdot \nabla) \Theta \right) \Lambda^{2\beta}\Theta dx
- \int_{\Omega} \left( (U \cdot \nabla) \theta_2 \right) \Lambda^{2\beta}\Theta dx.
\end{aligned}
\end{equation}
Using Young’s inequality and Poincar\'{e}’s inequality, we have
\begin{align*}
\int_{\Omega} \mathbb{P} (\Theta e_3) \cdot \Lambda^{2\alpha} U dx
\leq &
\frac{\nu}{6} \| \Lambda^{2\alpha} U(t) \|_{L^2(\Omega)}^{2} + \frac{3}{2\nu} \| \Theta(t) \|_{L^2(\Omega)}^{2} \\
\leq &
\frac{\nu}{6} \| \Lambda^{2\alpha} U(t) \|_{L^2(\Omega)}^{2} + C \|\Lambda^{\beta} \Theta(t) \|_{L^2(\Omega)}^{2}.
\end{align*}
By using H\"{o}lder’s inequality, as well as, Young’s inequality, we have  
\begin{align*}
\int_{\Omega} \mathbb{P} \left( (u_1 \cdot \nabla) U \right) \cdot \Lambda^{2\alpha}U dx
   & \leq \int_{\Omega}  |u_1| |\nabla U | |\Lambda^{2\alpha}U| dx \\
   & \leq \| u_1(t) \|_{L^{12}(\Omega)} \| \nabla U(t) \|_{L^{\frac{12}{5}}(\Omega)} \|\Lambda^{2\alpha} U(t) \|_{L^2(\Omega)} \\
   & \leq C\| \Lambda^{\frac{5}{4}}u_{1}(t) \|_{L^{2}(\Omega)} \| \Lambda^{\alpha}U(t) \|_{L^2(\Omega)} \|\Lambda^{2\alpha} U(t) \|_{L^2(\Omega)} \\
   & \leq \frac{\nu}{6} \| \Lambda^{2\alpha}U (t) \|_{L^{2}(\Omega)}^{2} + C \| \Lambda^{\frac{5}{4}} u_1(t) \|_{L^2(\Omega)}^{2} \| \Lambda^{\alpha}U(t) \|_{L^2(\Omega)}^{2}
\end{align*}
and it similarly follows that
\begin{align*}
  \int_{\Omega} \mathbb{P} \left( (U \cdot \nabla) u_2 \right) \cdot \Lambda^{2\alpha}U dx
   & \leq \int_{\Omega}  |U| |\nabla u_2| |\Lambda^{2\alpha}U| dx \\
   & \leq \| U(t) \|_{L^{12}(\Omega)} \| \nabla u_2(t) \|_{L^{\frac{12}{5}}(\Omega)} \|\Lambda^{2\alpha} U(t) \|_{L^2(\Omega)} \\
   & \leq C\| \Lambda^{\alpha}U(t) \|_{L^{2}(\Omega)} \| \Lambda^{\frac{5}{4}} u_2 (t) \|_{L^2(\Omega)} \|\Lambda^{2\alpha} U(t) \|_{L^2(\Omega)} \\
   & \leq \frac{\nu}{6} \| \Lambda^{2\alpha}U (t) \|_{L^{2}(\Omega)}^{2} + C \| \Lambda^{\frac{5}{4}} u_2(t) \|_{L^2(\Omega)}^{2} \| \Lambda^{\alpha}U(t) \|_{L^2(\Omega)}^{2}.
\end{align*}
Similarly, applying Gagliardo-Nirenberg inequality to the following expression, we can further derive that
\begin{align*}
  \int_{\Omega} \left( (u_1 \cdot \nabla) \Theta \right) \Lambda^{2\beta}\Theta dx
   & \leq \int_{\Omega}  |u_1| |\nabla\Theta| |\Lambda^{2\beta}\Theta| dx \\
   & \leq \| u_1(t) \|_{L^{12}(\Omega)} \| \nabla\Theta(t) \|_{L^{\frac{12}{5}}(\Omega)} \|\Lambda^{2\beta}\Theta(t) \|_{L^2(\Omega)} \\
   & \leq C \| \Lambda^{\frac{5}{4}} u_1 (t) \|_{L^2(\Omega)}
   \|\Lambda^{\frac{5}{8}}\Theta(t) \|_{L^2(\Omega)}^{\frac{16\beta-10}{16\beta-5}} \|\Lambda^{2\beta}\Theta(t) \|_{L^2(\Omega)}^{\frac{16\beta}{16\beta-5}}  \\
   & \leq \frac{\kappa}{4} \|\Lambda^{2\beta}\Theta(t) \|_{L^2(\Omega)}^2  + C \| \Lambda^{\frac{5}{4}} u_1 (t) \|_{L^2(\Omega)}^{\frac{16\beta-5}{8\beta-5}}
   \|\Lambda^{\frac{5}{8}}\Theta(t) \|_{L^2(\Omega)}^{2}\\
   & \leq \frac{\kappa}{4} \|\Lambda^{2\beta}\Theta(t) \|_{L^2(\Omega)}^2  + C \| \Lambda^{\frac{5}{4}} u_1 (t) \|_{L^2(\Omega)}^{\frac{16\beta-5}{8\beta-5}}
   \|\Lambda^{\beta}\Theta(t) \|_{L^2(\Omega)}^{2}
\end{align*}
Also we have
\begin{align*}
  \int_{\Omega} \left( (U \cdot \nabla) \theta_2 \right) \Lambda^{2\beta}\Theta dx
   & \leq \int_{\Omega}  |U| |\nabla \theta_2| |\Lambda^{2\beta}\Theta| dx \\
   & \leq \| U(t) \|_{L^{12}(\Omega)} \| \nabla \theta_2(t) \|_{L^{\frac{12}{5}}(\Omega)} \|\Lambda^{2\beta} \Theta(t) \|_{L^2(\Omega)} \\
   & \leq C\| \Lambda^{\alpha}U(t) \|_{L^{2}(\Omega)} \| \Lambda^{\frac{5}{4}} \theta_2 (t) \|_{L^2(\Omega)} \|\Lambda^{2\beta} \Theta(t) \|_{L^2(\Omega)} \\
   & \leq \frac{\kappa}{4} \| \Lambda^{2\beta} \Theta (t) \|_{L^{2}(\Omega)}^{2} + C \| \Lambda^{\frac{5}{4}} \theta_2(t) \|_{L^2(\Omega)}^{2} \| \Lambda^{\alpha}U(t) \|_{L^2(\Omega)}^{2}.
\end{align*}
Inserting the above five estimates back into \cref{eq0502}, we obtain
\begin{equation}\label{eq0503}
\begin{aligned}
& \frac{d}{dt}
\left( 
\| \Lambda^{\alpha}U(t) \|_{L^2(\Omega)}^{2} + \| \Lambda^{\beta}\Theta (t) \|_{L^2(\Omega)}^{2} 
\right)
+\nu \| \Lambda^{2\alpha}U(t) \|_{L^2(\Omega)}^{2} 
+\kappa \| \Lambda^{2\beta}\Theta(t) \|_{L^2(\Omega)}^{2} 
\\
\leq
& 
\gamma(t)
\left( 
\| \Lambda^{\alpha}U(t) \|_{L^2(\Omega)}^{2} + \| \Lambda^{\beta}\Theta (t) \|_{L^2(\Omega)}^{2}
\right),
\end{aligned}
\end{equation}
where
$$
\gamma(t)=
C \left( 1+
\| \Lambda^{\frac{5}{4}}u_1 (t) \|_{L^2(\Omega)}^{2} 
+\| \Lambda^{\frac{5}{4}}u_2 (t) \|_{L^2(\Omega)}^{2} 
+ \| \Lambda^{\frac{5}{4}} u_1 (t) \|_{L^2(\Omega)}^{\frac{16\beta-5}{8\beta-5}}
+ \| \Lambda^{\frac{5}{4}}\theta_2 (t) \|_{L^2(\Omega)}^{2} 
\right).
$$
Let $s \in (0, \ell)$. Applying Gronwall's inequality to \cref{eq0503} over $[\tau + s, t + s]$, we obtain
\begin{equation}\label{eq0504}
\begin{aligned}
&\| \Lambda^{\alpha}U(t+s) \|_{L^2(\Omega)}^{2} + \| \Lambda^{\beta}\Theta (t+s) \|_{L^2(\Omega)}^{2} \\
\leq & \left( \| \Lambda^{\alpha}U(\tau+s) \|_{L^2(\Omega)}^{2} + \| \Lambda^{\beta}\Theta (\tau+s) \|_{L^2(\Omega)}^{2} \right) \exp \left( \int_{\tau+s}^{t+\ell} \gamma(\zeta) d\zeta \right) \\
\leq & \left( \| \Lambda^{\alpha}U(\tau+s) \|_{L^2(\Omega)}^{2} + \| \Lambda^{\beta}\Theta (\tau+s) \|_{L^2(\Omega)}^{2} \right) \exp \left( \int_{\tau}^{t+\ell} \gamma(\zeta) d\zeta \right).
\end{aligned}
\end{equation}
By integrating both sides of \cref{eq0504} with respect to $s$ from $0$ to $\ell$, we have
\begin{equation}\label{eq0320}    
\begin{aligned}
&\int_{0}^{\ell} 
\left( \| \Lambda^{\alpha}U(t+s) \|_{L^2(\Omega)}^{2} + \| \Lambda^{\beta}\Theta (t+s) \|_{L^2(\Omega)}^{2} \right)
 ds \\
\leq &  \mathcal{K}_{\ell}(t,\tau)
\int_{0}^{\ell} 
\left( \| \Lambda^{\alpha}U(\tau+s) \|_{L^2(\Omega)}^{2} + \| \Lambda^{\beta}\Theta (\tau+s) \|_{L^2(\Omega)}^{2} \right)ds,
\end{aligned}
\end{equation}
where
$$
\mathcal{K}_{\ell}(t,\tau)= \exp \left( \int_{\tau}^{t+\ell} \gamma(\zeta) d\zeta \right)
$$
is a finite constant depending on $(u_1(\tau), \theta_1(\tau))$ and $(u_2(\tau), \theta_2(\tau))$, as follows from \cref{thm0203,cor01}. Hence, it follows that 
$$
\| L(t, \tau)\chi_1 - L(t, \tau)\chi_2 \|_{L^2(t, t+\ell; H_\sigma^{\alpha}(\Omega) \times H^{\beta}(\Omega))}^2 
\leq \mathcal{K}_\ell(t, \tau) \| \chi_1 - \chi_2 \|_{L^2(\tau, \tau+\ell; H_\sigma^{\alpha}(\Omega) \times H^{\beta}(\Omega))}^2,
$$
which gives the mapping $L(t, \tau) : X_\ell \to X_\ell$ is Lipschitz continuous on $\mathcal{B}_0^\ell(\tau)$ for all $t \geq \tau + \ell$.
\end{proof}

Using the invariance of \(B_1(t)\) together with \cref{lem0304,lem0307,lem0308}, we can infer that \(Y = \overline{L(t, \tau) \mathcal{B}_0^\ell(\tau)}^{X_\ell}\) is a positively invariant, uniformly pullback absorbing, compact subset of \(X_\ell\) whenever \(t - \tau \geq \tau_{4}\). Subsequently, we are able to derive the following result. 

\begin{theorem}\label{lem0309}
Assume that \cref{ass:forcing} holds. Then the process \(\{L(t, \tau)\}_{t \geq\tau}\) generated by the system in \cref{eq0101} admits a pullback attractor \(\hat{\mathcal{A}}_{\ell}=\{\mathcal {A}_{\ell}(t): t \in \mathbb{R}\}\) in \(X_{\ell}\) and \(e_{1}(\mathcal {A}_{\ell}(t-\ell)) \subset B_{1}(t)\) for any \(t \in \mathbb{R}\) ,where 
\[
e_{1}(\mathcal {A}_{\ell }(t-\ell ))=\{ e_{1}(\chi ):\chi \in \mathcal {A}_{\ell }(t-\ell )\}
\]
for any \(t \in \mathbb{R}\).
\end{theorem}

\section{The existence of pullback exponential attractors}\label{sec04}
\subsection{The existence of pullback exponential attractors in \texorpdfstring{$X_{\ell}$}{xl}}\label{subsec0401}
We construct a pullback exponential attractor $\hat{\mathcal{M}}_{\ell} = \left\{ \mathcal{M}_{\ell} (t): t \in \mathbb{R} \right\}$ in $X_{\ell}$ by combining the $\ell$-trajectory method and the smoothing property of the process $\{L(t, \tau)\}_{t \geq \tau}$.

\begin{lemma}\label{lem0401}
Assume that \cref{ass:forcing} holds. Then, for any $\tau \in \mathbb{R}$, the mapping $L(t, \tau) : X_\ell \to W_\ell$ is Lipschitz continuous on $\mathcal{B}_0^\ell(\tau) \subset X_\ell$ for all $t \geq \tau + \ell$.  
\end{lemma}
\begin{proof}
By \cref{defWl,ALlem}, it follows that $ W_\ell \subset\subset X_\ell $. For all fixed $\tau \in \mathbb{R}$ and any $\chi_{1}, \chi_{2} \in \mathcal{B}_{0}^{\ell}(\tau)$, let $L(t, \tau)\chi_1 = (u_1(t), \theta_1(t))$ and $L(t, \tau)\chi_2 = (u_2(t), \theta_2(t))$ for any fixed $t \geq \tau + \ell$. Let $(U, \Theta) = (u_1 - u_2, \theta_1 - \theta_2)$. Then we can derive that
\begin{equation}\label{eq030801}
\begin{aligned}
& \frac{d}{dt}
\left( 
\| \Lambda^{\alpha}U(t) \|_{L^2(\Omega)}^{2} + \| \Lambda^{\beta}\Theta (t) \|_{L^2(\Omega)}^{2} 
\right)
+\nu \| \Lambda^{2\alpha}U(t) \|_{L^2(\Omega)}^{2} 
+\kappa \| \Lambda^{2\beta}\Theta(t) \|_{L^2(\Omega)}^{2} 
\\
\leq
& 
\gamma(t)
\left( 
\| \Lambda^{\alpha}U(t) \|_{L^2(\Omega)}^{2} + \| \Lambda^{\beta}\Theta (t) \|_{L^2(\Omega)}^{2}
\right),
\end{aligned}
\end{equation}
where
$$
\gamma(t)=
C \left( 1+
\| \Lambda^{\frac{5}{4}}u_1 (t) \|_{L^2(\Omega)}^{2} 
+\| \Lambda^{\frac{5}{4}}u_2 (t) \|_{L^2(\Omega)}^{2} 
+ \| \Lambda^{\frac{5}{4}} u_1 (t) \|_{L^2(\Omega)}^{\frac{16\beta-5}{8\beta-5}}
+ \| \Lambda^{\frac{5}{4}}\theta_2 (t) \|_{L^2(\Omega)}^{2} 
\right).
$$
Consequently, for any $t \geq \tau + \ell$, integrating \cref{eq030801} from $t-s$ to $t+\ell$ with $s\in [0,\frac{\ell}{2}]$, we have
\begin{align*}
&
\| \Lambda^{\alpha}U(t+\ell) \|_{L^2(\Omega)}^{2} 
+ \| \Lambda^{\beta}\Theta (t+\ell) \|_{L^2(\Omega)}^{2} 
+\nu \int_{t-s}^{t+\ell} \| \Lambda^{2\alpha} U(\zeta) \|_{L^2(\Omega)}^{2} d\zeta\\
&
+\kappa \int_{t-s}^{t+\ell} \| \Lambda^{2\beta} \Theta (\zeta) \|_{L^2(\Omega)}^{2} d\zeta
\\
\leq
& \int_{t-s}^{t+\ell}
\gamma(\zeta)
\left( 
\| \Lambda^{\alpha}U(\zeta) \|_{L^2(\Omega)}^{2} + \| \Lambda^{\beta}\Theta(\zeta) \|_{L^2(\Omega)}^{2} 
\right) d\zeta 
+ \| \Lambda^{\alpha}U(t-s) \|_{L^2(\Omega)}^{2} 
+ \| \Lambda^{\beta}\Theta (t-s) \|_{L^2(\Omega)}^{2}.
\end{align*}
Employing the Gronwall Lemma, \cref{eq030801} yields that 
\begin{align*}
&\| \Lambda^{\alpha}U(t+\ell) \|_{L^2(\Omega)}^{2} 
+ \| \Lambda^{\beta}\Theta (t+\ell) \|_{L^2(\Omega)}^{2} \\
\leq&
\left( \| \Lambda^{\alpha}U(t-s) \|_{L^2(\Omega)}^{2} 
+ \| \Lambda^{\beta}\Theta (t-s) \|_{L^2(\Omega)}^{2} \right)
\exp \left(\int_{t-s}^{t+\ell}
\gamma(\zeta) d\zeta \right).
\end{align*}
For $\zeta \in [t-s,t+\ell]$, integrating \cref{eq030801} over the interval $[t-s,\zeta]$ gives us
\begin{align*}
&\| \Lambda^{\alpha}U(\zeta) \|_{L^2(\Omega)}^{2} 
+ \| \Lambda^{\beta}\Theta (\zeta) \|_{L^2(\Omega)}^{2} \\
&\leq
\left( \| \Lambda^{\alpha}U(t-s) \|_{L^2(\Omega)}^{2} 
+ \| \Lambda^{\beta}\Theta (t-s) \|_{L^2(\Omega)}^{2} \right)
\exp\left(\int_{t-s}^{\zeta}
\gamma(r) dr\right).
\end{align*}
Combining this estimate with the integrated inequality, we obtain
\begin{align*}
&\nu \int_{t-s}^{t+\ell} \| \Lambda^{2\alpha} U(\zeta) \|_{L^2(\Omega)}^{2} d\zeta
+\kappa \int_{t-s}^{t+\ell} \| \Lambda^{2\beta} \Theta (\zeta) \|_{L^2(\Omega)}^{2} d\zeta
\\
\leq
&\left( \| \Lambda^{\alpha}U(t-s) \|_{L^2(\Omega)}^{2} 
+ \| \Lambda^{\beta}\Theta (t-s) \|_{L^2(\Omega)}^{2} \right)
\left( \int_{t-\frac{\ell}{2}}^{t+\ell} \gamma(\zeta) d\zeta \right)
\exp \left(\int_{t-s}^{t+\ell}
\gamma(r) dr \right) \\
&+ \| \Lambda^{\alpha}U(t-s) \|_{L^2(\Omega)}^{2} 
+ \| \Lambda^{\beta}\Theta (t-s) \|_{L^2(\Omega)}^{2}.
\end{align*}
Hence, 
\begin{equation}\label{eq030802}
\begin{aligned}
&
\| \Lambda^{\alpha}U(t+\ell) \|_{L^2(\Omega)}^{2} 
+ \| \Lambda^{\beta}\Theta (t+\ell) \|_{L^2(\Omega)}^{2} \\
& +\nu \int_{t-s}^{t+\ell} \| \Lambda^{2\alpha} U(\zeta) \|_{L^2(\Omega)}^{2} d\zeta
+\kappa \int_{t-s}^{t+\ell} \| \Lambda^{2\beta} \Theta (\zeta) \|_{L^2(\Omega)}^{2} d\zeta
\\
\leq
& \mathcal{N}_{\ell}(t,\tau) \left( \| \Lambda^{\alpha}U(t-s) \|_{L^2(\Omega)}^{2} 
+ \| \Lambda^{\beta}\Theta (t-s) \|_{L^2(\Omega)}^{2} \right)
\exp \left(\int_{t-s}^{t+\ell}
\gamma(r) dr \right) \\
&+ \| \Lambda^{\alpha}U(t-s) \|_{L^2(\Omega)}^{2} 
+ \| \Lambda^{\beta}\Theta (t-s) \|_{L^2(\Omega)}^{2},
\end{aligned}
\end{equation}
where
\begin{align*}
\mathcal{N}_{\ell}(t,\tau) = \int_{\tau+\frac{\ell}{2}}^{t+\ell} \gamma(\zeta) d\zeta + 1.
\end{align*}
For any $t \geq \tau + \ell$, $s\in [0,\frac{\ell}{2}]$, and integrating \cref{eq030801} from $\tau+s$ to $t-s$, then we obtain
\begin{equation}\label{eq030803}
\begin{aligned}
& \| \Lambda^{\alpha}U(t-s) \|_{L^2(\Omega)}^{2} 
+ \| \Lambda^{\beta}\Theta (t-s) \|_{L^2(\Omega)}^{2} 
\\
\leq
& \int_{\tau+s}^{t-s}
\gamma(\zeta)
\left( 
\| \Lambda^{\alpha}U(\zeta) \|_{L^2(\Omega)}^{2} + \| \Lambda^{\beta}\Theta(\zeta) \|_{L^2(\Omega)}^{2} 
\right) d\zeta \\
&
+ \| \Lambda^{\alpha}U(\tau+s) \|_{L^2(\Omega)}^{2} 
+ \| \Lambda^{\beta}\Theta (\tau+s) \|_{L^2(\Omega)}^{2}.
\end{aligned}
\end{equation}
Using the Gronwall Lemma and \cref{eq030803}, we derive
\begin{equation}\label{eq030804}
\begin{aligned}
& \| \Lambda^{\alpha}U(t-s) \|_{L^2(\Omega)}^{2} 
+ \| \Lambda^{\beta}\Theta (t-s) \|_{L^2(\Omega)}^{2} \\
\leq
& 
\left( \| \Lambda^{\alpha}U(\tau+s) \|_{L^2(\Omega)}^{2} 
+ \| \Lambda^{\beta}\Theta (\tau+s) \|_{L^2(\Omega)}^{2} \right)
\exp \left(\int_{\tau+s}^{t-s}
\gamma(r) dr \right) \\
\leq
& 
\left( \| \Lambda^{\alpha}U(\tau+s) \|_{L^2(\Omega)}^{2} 
+ \| \Lambda^{\beta}\Theta (\tau+s) \|_{L^2(\Omega)}^{2} \right)
\exp \left(\int_{\tau}^{t-s}
\gamma(r) dr \right).
\end{aligned}
\end{equation}
According to \cref{eq030802,eq030804}, it follows that
\begin{equation}\label{eq030805}
\begin{aligned}
&
\nu \int_{0}^{\ell} \| \Lambda^{2\alpha} U(t+\zeta) \|_{L^2(\Omega)}^{2} d\zeta
+\kappa \int_{0}^{\ell} \| \Lambda^{2\beta} \Theta (t+\zeta) \|_{L^2(\Omega)}^{2} d\zeta
\\
\leq
& \mathcal{N}_{\ell}(t,\tau) \left( \| \Lambda^{\alpha}U(\tau+s) \|_{L^2(\Omega)}^{2} 
+ \| \Lambda^{\beta}\Theta (\tau+s) \|_{L^2(\Omega)}^{2} \right)
\exp \left(\int_{\tau}^{t+\ell}
\gamma(r) dr \right) \\
&+ \left( \| \Lambda^{\alpha}U(\tau+s) \|_{L^2(\Omega)}^{2} 
+ \| \Lambda^{\beta}\Theta (\tau+s) \|_{L^2(\Omega)}^{2} \right)
\exp \left(\int_{\tau}^{t-s}
\gamma(r) dr \right)\\
\leq
& 2 \mathcal{N}_{\ell}(t,\tau) \mathcal{O}_{\ell}(t,\tau) \left( \| \Lambda^{\alpha}U(\tau+s) \|_{L^2(\Omega)}^{2} 
+ \| \Lambda^{\beta}\Theta (\tau+s) \|_{L^2(\Omega)}^{2} \right),
\end{aligned}
\end{equation}
where
\begin{align*}
\mathcal{N}_{\ell}(t,\tau) = \int_{\tau+\frac{\ell}{2}}^{t+\ell} \gamma(\zeta) d\zeta + 1,\quad
\mathcal{O}_{\ell}(t,\tau) = \exp \left(\int_{\tau}^{t+\ell}
\gamma(r) dr \right).
\end{align*}
Integrating \cref{eq030805} with respect to $s$ over $(0,\frac{\ell}{2})$, we obtain
\begin{equation}\label{eq030806}
\begin{aligned}
&
 \int_{0}^{\ell} \| \Lambda^{2\alpha} U(t+\zeta) \|_{L^2(\Omega)}^{2} d\zeta
+ \int_{0}^{\ell} \| \Lambda^{2\beta} \Theta (t+\zeta) \|_{L^2(\Omega)}^{2} d\zeta
\\
\leq
& \frac{4 \mathcal{N}_{\ell}(t,\tau) \mathcal{O}_{\ell}(t,\tau) }{\min\{\nu,\kappa\}\ell} 
\int_{0}^{\frac{\ell}{2}}
\left( \| \Lambda^{\alpha}U(\tau+s) \|_{L^2(\Omega)}^{2} 
+ \| \Lambda^{\beta}\Theta (\tau+s) \|_{L^2(\Omega)}^{2} \right) ds,
\end{aligned}
\end{equation}
Given the boundedness of $\mathcal{N}_{\ell}(t,\tau)$ and $\mathcal{O}_{\ell}(t,\tau)$ for each fixed $t \in [\tau + \ell, +\infty)$, we deduce that
\begin{equation}\label{eq030807}
\begin{aligned}
&
\int_{0}^{\ell} 
\left( \| \Lambda^{2\alpha} U(t+\zeta) \|_{L^2(\Omega)}^{2} + \| \Lambda^{2\beta} \Theta (t+\zeta) \|_{L^2(\Omega)}^{2}  \right) d\zeta
\\
\leq
& C(\tau,t,\ell,\nu,\kappa)
\int_{0}^{\ell}
\left( \| \Lambda^{\alpha}U(\tau+s) \|_{L^2(\Omega)}^{2} 
+ \| \Lambda^{\beta}\Theta (\tau+s) \|_{L^2(\Omega)}^{2} \right) ds,
\end{aligned}
\end{equation}
and hence
$$
\| L(t, \tau)\chi_1 - L(t, \tau)\chi_2 \|_{L^2(t, t+\ell; H_\sigma^{2\alpha}(\Omega) \times H^{2\beta}(\Omega)}
\leq C(\tau,t,\ell,\nu,\kappa) \| \chi_1 - \chi_2 \|_{L^2(\tau, \tau+\ell; H_\sigma^{\alpha}(\Omega) \times H^{\beta}(\Omega))},
$$
for all $\chi_1, \chi_2 \in \mathcal{B}_0^\ell(\tau)$ and any $t \geq \tau + \ell$.

Additionally, for any $t \geq \tau + \ell$, integrating from $t-s$ to $t+r$ with $s\in [0,\frac{\ell}{2}]$ and $r\in [0,\ell]$, we get
\begin{align*}
&\| \Lambda^{\alpha}U(t+r) \|_{L^2(\Omega)}^{2} 
+ \| \Lambda^{\beta}\Theta (t+r) \|_{L^2(\Omega)}^{2} \\
\leq&
\left( \| \Lambda^{\alpha}U(t-s) \|_{L^2(\Omega)}^{2} 
+ \| \Lambda^{\beta}\Theta (t-s) \|_{L^2(\Omega)}^{2} \right)
\exp\left(\int_{t-s}^{t+r}
\gamma(r) dr\right)\\
\leq
& 
\left( \| \Lambda^{\alpha}U(\tau+s) \|_{L^2(\Omega)}^{2} 
+ \| \Lambda^{\beta}\Theta (\tau+s) \|_{L^2(\Omega)}^{2} \right)
\exp \left(\int_{\tau}^{t-s}
\gamma(r) dr \right)\exp\left(\int_{t-s}^{t+r}
\gamma(r) dr\right)\\
\leq
& 
\left( \| \Lambda^{\alpha}U(\tau+s) \|_{L^2(\Omega)}^{2} 
+ \| \Lambda^{\beta}\Theta (\tau+s) \|_{L^2(\Omega)}^{2} \right)
\exp \left(\int_{\tau}^{t+r}
\gamma(r) dr \right)\\
\leq &
\mathcal{O}_{\ell}(t,\tau)\left( \| \Lambda^{\alpha}U(\tau+s) \|_{L^2(\Omega)}^{2} 
+ \| \Lambda^{\beta}\Theta (\tau+s) \|_{L^2(\Omega)}^{2} \right).
\end{align*}
Integrating the above inequality with respect to $s$ from $0$ to $\frac{\ell}{2}$ yields
\begin{equation}\label{eq030808}
\begin{aligned}
&  \| \Lambda^{\alpha}U(t+r) \|_{L^2(\Omega)}^{2} 
+ \| \Lambda^{\beta}\Theta (t+r) \|_{L^2(\Omega)}^{2} \\
\leq & \frac{2}{\ell}
\mathcal{O}_{\ell}(t,\tau)\int_{0}^{\frac{\ell}{2}} \left( \| \Lambda^{\alpha}U(\tau+s) \|_{L^2(\Omega)}^{2} 
+ \| \Lambda^{\beta}\Theta (\tau+s) \|_{L^2(\Omega)}^{2} \right) ds \\
\leq & \frac{2}{\ell}
\mathcal{O}_{\ell}(t,\tau)\int_{0}^{\ell} \left( \| \Lambda^{\alpha}U(\tau+s) \|_{L^2(\Omega)}^{2} 
+ \| \Lambda^{\beta}\Theta (\tau+s) \|_{L^2(\Omega)}^{2} \right) ds.
\end{aligned}
\end{equation}
Leveraging \cref{eq0501}, we derive that 
\begin{equation*}
\begin{cases}
U_t = - \nu \Lambda^{2\alpha} U - \mathbb{P} \left( (u_1 \cdot \nabla) U \right) - \mathbb{P} \left( (U \cdot \nabla) u_2 \right) + \mathbb{P} (\Theta e_3), \\
\Theta_t = - \kappa \Lambda^{2\beta} \Theta - (u_1 \cdot \nabla) \Theta - (U \cdot \nabla) \theta_2.
\end{cases}
\end{equation*}
By H\"{o}lder’s inequality, we deduce that
\begin{equation}\label{eq030809}
\begin{aligned}
&\| U_{t} \|_{L^{2}(\Omega)}\\
\leq
& \nu \| \Lambda^{2\alpha} U \|_{L^2(\Omega)}
+ \| (u_1 \cdot \nabla) U \|_{L^{2}(\Omega)}
+ \| (U \cdot \nabla) u_2 \|_{L^{2}(\Omega)}
+ \| \Theta \|_{L^{2}(\Omega)} 
\\ 
\leq
& \nu \| \Lambda^{2\alpha} U \|_{L^2(\Omega)}
+ \| u_1  \|_{L^{12}(\Omega)}
\| \nabla U \|_{L^{\frac{12}{5}}(\Omega)}
+ \| U \|_{L^{12}(\Omega)}
\| \nabla  u_2 \|_{L^{\frac{12}{5}}(\Omega)}
+ \| \Theta \|_{L^{2}(\Omega)} 
\\ 
\leq
& \nu \| \Lambda^{2\alpha} U \|_{L^2(\Omega)}
+ C\| \Lambda^{\alpha}u_1  \|_{L^{2}(\Omega)}
\| \Lambda^{\alpha} U \|_{L^{2}(\Omega)}
+ C\| \Lambda^{\alpha}U \|_{L^{2}(\Omega)}
\| \Lambda^{\alpha} u_2 \|_{L^{2}(\Omega)}
+ \| \Theta \|_{L^{2}(\Omega)} 
\\ 
\end{aligned}
\end{equation}
and
\begin{equation}\label{eq030810}
\begin{aligned}
\| \Theta_t \|_{L^2(\Omega)} 
\leq
& \kappa \| \Lambda^{2\beta} \Theta  \|_{L^2(\Omega)}
+ 
\| (u_1 \cdot \nabla) \Theta  \|_{L^{2}(\Omega)}
+ \| (U \cdot \nabla) \theta_2 \|_{L^{2}(\Omega)}\\
\leq
& \kappa \| \Lambda^{2\beta} \Theta  \|_{L^2(\Omega)}
+ 
\| u_1  \|_{L^{12}(\Omega)}
\| \nabla \Theta \|_{L^{\frac{12}{5}}(\Omega)}
+ \| U \|_{L^{12}(\Omega)}
\| \nabla  \theta_2 \|_{L^{\frac{12}{5}}(\Omega)}\\
\leq
& \kappa \| \Lambda^{2\beta} \Theta  \|_{L^2(\Omega)}
+ 
\| \Lambda^{\alpha}u_1  \|_{L^{2}(\Omega)}
\| \Lambda^{2\beta}\Theta \|_{L^{2}(\Omega)}
+ \| \Lambda^{\alpha}U \|_{L^{2}(\Omega)}
\|  \Lambda^{2\beta} \theta_2 \|_{L^{2}(\Omega)}.
\end{aligned}
\end{equation}
Then we have
\begin{equation}\label{eq030811}
\begin{aligned}
\| U_{t} \|_{L^{2}(t,t+\ell;L^{2}(\Omega))}
\leq
& \nu \| U \|_{L^{2}(t,t+\ell;H_\sigma^{2\alpha}(\Omega))}
+ C\| u_1 \|_{L^{\infty}(t,t+\ell;H_\sigma^{\alpha}(\Omega))}
\| U \|_{L^{2}(t,t+\ell;H_\sigma^{\alpha}(\Omega))} \\
&+ C
\| u_2 \|_{L^{\infty}(t,t+\ell;H_\sigma^{\alpha}(\Omega))}
\| U \|_{L^{2}(t,t+\ell;H_\sigma^{\alpha}(\Omega))}
+ \| \Theta \|_{L^{2}(t,t+\ell;L^{2}(\Omega))} 
\\ 
\leq
& \nu \| U \|_{L^{2}(t,t+\ell;H_\sigma^{2\alpha}(\Omega))}
+ C\| u_1 \|_{L^{\infty}(t,t+\ell;H_\sigma^{\alpha}(\Omega))}
\| U \|_{L^{2}(t,t+\ell;H_\sigma^{2\alpha}(\Omega))} \\
&+ C
\| u_2 \|_{L^{\infty}(t,t+\ell;H_\sigma^{\alpha}(\Omega))}
\| U \|_{L^{2}(t,t+\ell;H_\sigma^{2\alpha}(\Omega))}
+ C \| \Theta \|_{L^{2}(t,t+\ell;H^{2\beta}(\Omega))} 
\end{aligned}
\end{equation}
and
\begin{equation}\label{eq030812}
\begin{aligned}
\| \Theta_t \|_{L^{2}(t,t+\ell;L^{2}(\Omega))} 
\leq
& \kappa \| \Theta  \|_{L^{2}(t,t+\ell;H^{2\beta}(\Omega))}
+ 
\| u_1 \|_{L^{\infty}(t,t+\ell;H_\sigma^{\alpha}(\Omega))}
\| \Theta  \|_{L^{2}(t,t+\ell;H^{2\beta}(\Omega))} \\
&
+ \|U \|_{L^{\infty}(t,t+\ell;H_\sigma^{\alpha}(\Omega))}
\| \theta_2 \|_{L^{2}(t,t+\ell;H^{2\beta}(\Omega))}.
\end{aligned}
\end{equation}
By \cref{thm0203}, the norms $\| u_1 \|_{L^{\infty}(t,t+\ell;H_\sigma^{\alpha}(\Omega))},\ \| u_2 \|_{L^{\infty}(t,t+\ell;H_\sigma^{\alpha}(\Omega))}$, and $\| \theta_2 \|_{L^{2}(t,t+\ell;H^{2\beta}(\Omega))}$ are uniformly bounded. Combining \cref{eq030807,eq030808,eq030811,eq030812}, we obtain 
\begin{equation}\label{eq030813}
\begin{aligned}
&
\int_{0}^{\ell} 
\left( \| U_{t}(t+\zeta) \|_{L^2(\Omega)}^{2} + \| \Theta_{t}(t+\zeta) \|_{L^2(\Omega)}^{2}  \right) d\zeta
\\
\leq
& C(\tau,t,\ell,\nu,\kappa)
\int_{0}^{\ell}
\left( \| \Lambda^{\alpha}U(\tau+s) \|_{L^2(\Omega)}^{2} 
+ \| \Lambda^{\beta}\Theta (\tau+s) \|_{L^2(\Omega)}^{2} \right) ds,
\end{aligned}
\end{equation}
and hence
$$
\|\left( L(t, \tau)\chi_1 - L(t, \tau)\chi_2  \right)_{t}\|_{L^2(t, t+\ell; H_{1}\times H_{2})}
\leq C(\tau,t,\ell,\nu,\kappa) \| \chi_1 - \chi_2 \|_{L^2(\tau, \tau+\ell; H_\sigma^{\alpha}(\Omega) \times H^{\beta}(\Omega))},
$$
for any $\chi_1, \chi_2 \in \mathcal{B}_0^\ell(\tau)$ and any $t \geq \tau + \ell$.
\end{proof}

Employing the \(\ell\)-trajectory method and the smoothing property of the process \(\{ L(t, \tau) \}_{t \geq \tau}\) generated by \cref{eq0101}, we establish the existence of pullback exponential attractors.

\begin{theorem}\label{thm0402}
Assume that \cref{ass:forcing} holds. For any $\eta \in (0,\frac{1}{2})$, there exists a pullback exponential attractor $\hat{\mathcal{M}}_{\ell} = \hat{\mathcal{M}}_{\ell}^{\eta} = \left\{ \mathcal{M}_{\ell} (t): t \in \mathbb{R} \right\}$ for the process $ \{L(t,\tau)\}_{t\geq \tau} $ generated by \cref{eq0101}. The sections $\mathcal{M}_{\ell}(t)$ are compact subsets of $X_{\ell}$, and their fractal dimension in $ L^2(\tau,\tau+\ell;H_\sigma^{\alpha}(\Omega) \times H^{\beta}(\Omega))$ can be estimated by
\[
\dim_F(\mathcal{M}_{\ell}(t)) \leq \log_{\frac{1}{2\eta}}\bigl(N_{\eta}^{X_{\ell}}(B_{W_{\ell}}(0;1))\bigr),
\]
where $B_X(x_0, R)$ denotes the $R$-ball in a normed space $X$ centered at $x_0$, and $N_{\eta}^{X_{\ell}}(B_{W_{\ell}}(0, 1))$ denotes the minimal number of $\eta$-balls in $X_{\ell}$ required to cover the unit ball $B_{W_{\ell}}(0, 1)$ in $W_{\ell}$.
\end{theorem}

\begin{proof}
Following the proof framework of \cite{LY2020,AT2021,YouB2023,YouB2021}, and combining the results obtained earlier in this paper with the $\ell$-trajectory method, the existence of a pullback exponential attractor in \(X_{\ell}\) — corresponding to the process \(\{L(t, \tau)\}_{t \geq \tau}\) generated by the 3D non-autonomous Boussinesq equations with fractional Laplacian — can be constructed and proven.

First, by utilizing the local $L^2$-integrability of the external force terms and the uniform boundedness estimates of solutions in \cref{lem0304}, it is proven that the system admits a pullback absorbing set \(\hat{\mathcal{B}}_0^{\ell} = \{\mathcal{B}_0^{\ell}(t) \mid t \in \mathbb{R}\}\) in \(X_{\ell}\). Meanwhile, it is verified that the process \(L(t, \tau)\) possesses smoothness on \(\mathcal{B}_0^{\ell}(\tau)\) in \cref{lem0401}. Moreover, the compact embedding relation \(W_{\ell} \subset \subset X_{\ell}\) holds for \(W_{\ell}\), which lays the foundation for subsequent finite covering and precompactness analysis.

Next, focusing on discrete time nodes (where \(t = kt_1\) and \(\tau = (k-1)t_1\) with \(k \in \mathbb{Z}\) and \(t_1 > 0\) as a fixed time step), a sequence of center sets \(V^n(k)\) is constructed via mathematical induction: in the initial step, the single-element center set \(V^0(k)\) is defined; in each subsequent step, relying on the compact embedding property of \(W_{\ell}\), \(L(kt_1, (k-n)t_1)\mathcal{B}_0^{\ell}((k-n)t_1)\) is covered with a finite number of \(X_{\ell}\)-balls, where the covering radius decays exponentially in the form of \((2\eta)^n R\) ($\eta \in (0, \frac{1}{2})$ and \(R > 0\) denotes a uniform bound). Furthermore, the precompact set \(\tilde{\mathcal{E}}_{\ell}(k)\) is defined as the union of center sets of all orders, and its precompactness in \(X_{\ell}\), the uniform boundedness of its fractal dimension, and its exponential attractivity for bounded subsets of \(X_{\ell}\) are verified by using \cref{lem0307}.  

Finally, the discrete-time results are extended to the continuous time process. For any \(t \in [kt_1, (k+1)t_1)\), \(\mathcal{M} _{\ell}(t)\) is defined as the closure of \(L(t, kt_1)\tilde{\mathcal{E}}_{\ell}(k)\) in \(X_{\ell}\). Leveraging the continuity of the process \(L(t, \tau)\), the compactness of \( \mathcal{M} _{\ell}(t)\) (preservation of compactness by closure), its positive semi-invariance (i.e., \(L(t, s)\mathcal{M} _{\ell}(s) \subset \mathcal{M} _{\ell}(t)\) holds for all \(t \geq s\)), and its exponential attractivity for bounded subsets in the continuous-time setting are sequentially verified. Ultimately, it is confirmed that the family of sets \(\hat{\mathcal{M} }_{\ell} = \{\mathcal{M} _{\ell}(t) \mid t \in \mathbb{R}\}\) satisfies all the defining conditions of a pullback exponential attractor, thereby completing the proof of the existence of the pullback exponential attractor for the 3D Boussinesq system in \(X_{\ell}\).
\end{proof}

\subsection{The existence of pullback exponential attractors in \texorpdfstring{$H_\sigma^{\alpha}(\Omega) \times H^{\beta}(\Omega)$}{HH}}\label{subsec0402}
\begin{lemma}\label{lem0403}
Assume that \cref{ass:forcing} holds. Therefore, the mapping $e_1: \mathcal{B}_0^\ell(\tau-\ell) \to B_1(\tau)=e_1(\mathcal{B}_0^\ell(\tau-\ell))$ is Lipschitz continuous for each fixed $\tau \in \mathbb{R}$. Specifically, for any $\ell$-trajectories $\chi_1, \chi_2 \in \mathcal{B}_0^\ell(\tau-\ell)$, there exists a constant $\mu > 0 $ (depending on $\ell$) satisfying
\begin{align*}
\| e_{1}(\chi_{1}) - e_{1}(\chi_{2}) \|_{H_\sigma^{\alpha}(\Omega) \times H^{\beta}(\Omega)}^{2} 
\leq
\mu
\int_{0}^{\ell}
\| \chi_{1}(\tau+\zeta) - \chi_{2}(\tau+\zeta) \|_{H_\sigma^{\alpha}(\Omega) \times H^{\beta}(\Omega)}^{2} 
d\zeta.
\end{align*}
\end{lemma}
\begin{proof}
For all fixed $\tau \in \mathbb{R}$ and any $\chi_{1}, \chi_{2} \in \mathcal{B}_{0}^{\ell}(\tau)$, let $L(t, \tau)\chi_1 = (u_1(t), \theta_1(t))$ and $L(t, \tau)\chi_2 = (u_2(t), \theta_2(t))$ for any fixed $t \geq \tau + \ell$. Let $(U, \Theta) = (u_1 - u_2, \theta_1 - \theta_2)$. Then we can derive that
\begin{equation}\label{eq050201}
\begin{aligned}
& \frac{d}{dt}
\left( 
\| \Lambda^{\alpha}U(t) \|_{L^2(\Omega)}^{2} + \| \Lambda^{\beta}\Theta (t) \|_{L^2(\Omega)}^{2} 
\right)
+\nu \| \Lambda^{2\alpha}U(t) \|_{L^2(\Omega)}^{2} 
+\kappa \| \Lambda^{2\beta}\Theta(t) \|_{L^2(\Omega)}^{2} 
\\
\leq
& 
\gamma(t)
\left( 
\| \Lambda^{\alpha}U(t) \|_{L^2(\Omega)}^{2} + \| \Lambda^{\beta}\Theta (t) \|_{L^2(\Omega)}^{2}
\right),
\end{aligned}
\end{equation}
where
$$
\gamma(t)=
C \left( 1+
\| \Lambda^{\frac{5}{4}}u_1 (t) \|_{L^2(\Omega)}^{2} 
+\| \Lambda^{\frac{5}{4}}u_2 (t) \|_{L^2(\Omega)}^{2} 
+ \| \Lambda^{\frac{5}{4}} u_1 (t) \|_{L^2(\Omega)}^{\frac{16\beta-5}{8\beta-5}}
+ \| \Lambda^{\frac{5}{4}}\theta_2 (t) \|_{L^2(\Omega)}^{2} 
\right).
$$
For any $\tau \in \mathbb{R}$ and $\zeta\in(0,\ell)$, by applying Gronwall’s lemma to \cref{eq050201}, we note that
\begin{align*}
&\| \Lambda^{\alpha}U(\tau+\ell) \|_{L^2(\Omega)}^{2} 
+ \| \Lambda^{\beta}\Theta (\tau+\ell) \|_{L^2(\Omega)}^{2} \\
\leq&
\left( \| \Lambda^{\alpha}U(\tau+\zeta) \|_{L^2(\Omega)}^{2} 
+ \| \Lambda^{\beta}\Theta (\tau+\zeta) \|_{L^2(\Omega)}^{2} \right)
\exp \left(\int_{\tau+\zeta}^{\tau+\ell}
\gamma(s) ds \right)\\
\leq&
\left( \| \Lambda^{\alpha}U(\tau+\zeta) \|_{L^2(\Omega)}^{2} 
+ \| \Lambda^{\beta}\Theta (\tau+\zeta) \|_{L^2(\Omega)}^{2} \right)
\exp \left(\int_{\tau}^{\tau+\ell}
\gamma(s) ds \right).
\end{align*}
Integrating the above inequality with respect to $\zeta$ over $(0,\ell)$, we obtain
\begin{equation}\label{eq050202}
\begin{aligned}
&\| \Lambda^{\alpha}U(\tau+\ell) \|_{L^2(\Omega)}^{2} 
+ \| \Lambda^{\beta}\Theta (\tau+\ell) \|_{L^2(\Omega)}^{2} \\
\leq& \frac{1}{\ell}
\exp \left(\int_{\tau}^{\tau+\ell}
\gamma(s) ds \right)\int_{0}^{\ell}\left( \| \Lambda^{\alpha}U(\tau+\zeta) \|_{L^2(\Omega)}^{2} 
+ \| \Lambda^{\beta}\Theta (\tau+\zeta) \|_{L^2(\Omega)}^{2} \right)d\zeta
\\
\leq& \mu
\int_{0}^{\ell}\left(\| \Lambda^{\alpha}U(\tau+\zeta) \|_{L^2(\Omega)}^{2} 
+ \| \Lambda^{\beta}\Theta (\tau+\zeta) \|_{L^2(\Omega)}^{2} \right)d\zeta,
\end{aligned}
\end{equation}
where $\mu=\frac{1}{\ell}\mathcal{O}_{\ell}(\tau)$, $\mathcal{O}_{\ell}(\tau) = \exp \left(\int_{\tau}^{\tau+\ell}
\gamma(r) dr \right)$ is a finite constant depending on $(u_1(\tau), \theta_1(\tau))$ and $(u_2(\tau), \theta_2(\tau))$, as follows from \cref{thm0203}. 

Consequently, appealing to \cref{eq050202}, we obtain that $e_1: \mathcal{B}_0^\ell(\tau-\ell) \to B_1(\tau)$ is Lipschitz continuous.
\end{proof}

\begin{theorem}\label{thm0404}
Assume that \cref{ass:forcing} holds. For any \(\eta \in (0, \frac{1}{2})\), there exists a pullback exponential attractor $\hat{\mathcal{M}} = \hat{\mathcal{M}}^\eta = \left\{ \mathcal{M} (t): t \in \mathbb{R} \right\} = \left\{ e_{1}(\mathcal{M}_{\ell} (t-\ell)) : t \in \mathbb{R} \right\} $ for the process \(\{ U(t, \tau) \}_{t \geq \tau}\) generated by \cref{eq0101} in $H_\sigma^{\alpha}(\Omega) \times H^{\beta}(\Omega)$.
\end{theorem}

\begin{proof}
The proof of \cref{thm0404} follows the framework from \cite{LY2020,AT2021,YouB2023,YouB2021} and \cref{weishuguji,thm0402,lem0403}.
\end{proof}

By \cref{thm0404}, we also obtain the following result.
\begin{theorem}
Assume that \cref{ass:forcing} holds. There exists a pullback attractor \(\hat{\mathcal{A}} = \{\mathcal{A}(t) : t \in \mathbb{R}\} = \{e_1(\mathcal{A}_{\ell}(t-\ell)) : t \in \mathbb{R}\}\) for the process \(\{U(t, \tau)\}_{t \geq \tau}\) generated by \cref{eq0101} in $H_\sigma^{\alpha}(\Omega) \times H^{\beta}(\Omega)$, and it satisfies 
$$\mathcal{A}(t) \subset \mathcal{M}(t)$$
for all \(t \in \mathbb{R}\). Here, \(\mathcal{A}_{\ell}(t-\ell)\) denotes the section at time \(t-\ell\) of the pullback attractor \(\hat{\mathcal{A}}_{\ell} = \{\mathcal{A}_{\ell}(s) : s \in \mathbb{R}\}\), which was established in \cref{lem0309} for the process \(\{L(t, \tau)\}_{t \geq \tau}\) associated with \cref{eq0101} in \(X_{\ell}\).
\end{theorem}
\begin{proof}
From \cref{weishuguji,lem0309,lem0403}, we infer the existence of the pullback attractor \(\hat{\mathcal{A}} = \{\mathcal{A}(t) \mid t \in \mathbb{R} \} = \{ e_1(\mathcal{A}_\ell(t-\ell)) \mid t \in \mathbb{R} \}\) for the process \(\{ U(t, \tau) \}_{t \geq \tau}\) generated by \cref{eq0101} in $H_\sigma^{\alpha}(\Omega) \times H^{\beta}(\Omega)$. From \cref{weishudingyi,lahuizhishudingyi,thm0402}, we further conclude that each section \(\mathcal{M}(t)\) of the pullback exponential attractor \(\hat{\mathcal{M}} = \{ \mathcal{M}(t) \mid t \in \mathbb{R} \}\) contains the corresponding section \(\mathcal{A}(t)\) of the pullback attractor \(\hat{\mathcal{A}}\), where \(\mathcal{A}_\ell(t-\ell)\) denotes the section of the pullback attractor \(\hat{\mathcal{A}}_\ell = \{ \mathcal{A}_\ell(t) \mid t \in \mathbb{R} \}\) — established in \cref{lem0309} for the process \(\{ L(t, \tau) \}_{t \geq \tau}\) generated by \cref{eq0101} in \(X_\ell\).
\end{proof}

\begin{remark}
Assume that \cref{ass:forcing} holds. If we can establish the Hölder continuity in time of the process \(\{U(t, \tau)\}_{t \geq \tau}\) generated by \cref{eq0101} in $H_\sigma^{\alpha}(\Omega) \times H^{\beta}(\Omega)$, then the pullback exponential attractor \(\hat{\mathcal{M}} = \{\mathcal{M}(t) : t \in \mathbb{R}\}\) for this process can be constructed as follows:
\[
\mathcal{M}_{\ell}(t) = \bigcup_{s \in [\tau, \tau+\ell]} L(t, s) \tilde{\mathcal{E}}_{\ell}(s),
\]
and
\[
\mathcal{M}(t) = e_1\left( \mathcal{M}_{\ell}(t-\ell) \right).
\]
\end{remark}

\section{Acknowledgement}
The authors sincerely thank the anonymous reviewers. Their insightful comments and constructive suggestions have greatly enhanced the quality of this manuscript.

\section*{Funding}
This work was supported by the National Natural Science Foundation of China (Nos. 12271293) and Natural Science Foundation of Shandong Province (No. ZR2023MA002, No. ZR2024MA069) and the project of Youth Innovation Team of Universities of Shandong Province (No. 2023KJ204) and Yunnan Fundamental Research Projects (No. 202401AT070411).

\section*{Conflict of interest}
The authors declares that they have no conflict of interest.

\end{document}